\documentclass[12pt]{article}
\usepackage{amsmath}
\usepackage{amssymb}
\usepackage{amsthm}
\usepackage{alltt}
\usepackage{amscd}
\usepackage{caption}
\usepackage[mathscr]{eucal}
\usepackage{mathrsfs}
\usepackage[all]{xy}
\usepackage{verbatim}
\usepackage{moreverb}
\usepackage{graphicx}
\DeclareMathAlphabet{\mathpzc}{OT1}{pzc}{m}{it}
\theoremstyle{plain}

\newtheorem{theorem}{Theorem}

\newcommand{\beq}{\begin{equation}}
\newcommand{\eeq}{\end{equation}}
\newcommand{\bc}{}

\newtheorem{cor}[theorem]{Corollary}

\newtheorem{lemma}[theorem]{Lemma}

\newtheorem*{conjecture}{Conjecture}

\numberwithin{equation}{section}

\numberwithin{theorem}{section}
\numberwithin{figure}{section}

\theoremstyle{definition}

\begin{document}
{\center\bf\large
Average Error for Spectral Asymptotics\\ on Surfaces\\}
\renewcommand{\thefootnote}{\fnsymbol{footnote}}
\noindent Robert S. Strichartz\footnote{Research supported by the
National Science Foundation, grant DMS-1162045\\

\underbar{Keywords}: Spectral asymptotics, Laplacian, average error, surfaces of constant curvature, cone point singularities, almost periodic functions\\

\noindent \underbar{Mathematics Subject Classification} (2010): 47 A 10, 58 C 40, 58 J 50}\\
Math Department\\
Malott Hall\\
Cornell University\\
Ithaca, NY 14853\\
str@math.cornell.edu\\

 \begin{abstract}
Let $N(t)$ denote the eigenvalue counting function of the Laplacian on
a compact surface of constant nonnegative curvature, with or without
boundary.  We define a refined asymptotic formula $\widetilde
N(t)=At+Bt^{1/2}+C$, where the constants are expressed in terms of the
geometry of the surface and its boundary, and consider the average
error $A(t)=\frac 1 t \int^t_0 D(s)\,ds$ for $D(t)=N(t)-\widetilde
N(t)$.  We present a conjecture for the asymptotic behavior of $A(t)$,
and study some examples that support the conjecture.\\
\end{abstract}

{\center ``The mills of God grind slowly, yet they grind exceeding small.''\\}
{\hfill Proverb\\}

\section{Introduction}

For any positive self-adjoint operator with discrete spectrum
\[0\le\lambda_1\le\lambda_2\le\lambda_3\le\cdots\rightarrow\infty\qquad
\text {(repeated according to multiplicity)}\]
we consider the eigenvalue counting function
\beq N(t)=\#\{\lambda_j\le t\}.\eeq
Often there is a predicted asymptotic approximation $\widetilde N(t)$,
and the problem arises to estimate the error
\beq D(t)=N(t)-\widetilde N(t).\eeq
Here we investigate some examples where the average error
\beq A(t)=\frac 1 t \int_0^t D(s)\,ds\eeq
is better behaved, following up on [JS].  In our examples we deal with
$-\Delta$ on a compact surface $S$ which is either flat or has
constant positive curvature.  The Weyl asymptotic formula has
$N(t)=At+\text o(t)$ where the constant $A$ is given by $A=\frac 1
{4\pi} \text{Area}(S)$.  If $S$ has a reasonable boundary then there
is a second term in the asymptotics, $N(t)=At+Bt^{1/2}+\text
o(t^{1/2})$, where the constant $B$ depends on the boundary
conditions.  If we impose Neumann boundary conditions on all of
$\partial S$, then $B=\frac 1 {4\pi}\text{length}(\partial S)$, and if
we impose Dirichlet boundary conditions then $B=-\frac 1
{4\pi}\text{length}(\partial S)$.  We may also consider mixed boundary
conditions, splitting $\partial S= \partial S_N\cup \partial S_D$ and
imposing Neumann boundary conditions on $\partial S_N$ and Dirichlet
boundary conditions on $\partial S_D$.  In that case $B=\frac 1
{4\pi}\text{length}(\partial S_N)-\frac 1 {4\pi}\text{length}(\partial
S_D)$.\\

For our purposes we will need a more refined asymptotics with an
additional constant term:
\beq \widetilde N(t)=At+Bt^{1/2}+C\eeq
where $A$ and $B$ are as above and $C$ is given as follows.  We assume
that $\partial S$ is piecewise smooth, with a finite number of corners
with angles $\{\theta_j\}$.  We will write $C=C_1+C_2+C_3$, where
$C_1$ is the contribution from the corners, $C_2$ is the contribution
from the curvature of the smooth arcs in $\partial S$, and $C_3$ is
the contribution from the curvature of $S$.  Thus $C_3$ will be zero
if $S$ is flat, and \beq
C_3=\frac 1 {12\pi} K_2(S),\eeq
where $K_2(S)$ is the total curvature, the integral of the curvature
over $S$, in the general case.  Thus for any sphere $K(S)=4\pi$ and
$C_3=\frac 1 3$.  For $C_2$ we will take the integral of the curvature
$K_1$ of the boundary,
\beq C_2=\frac 1 {12\pi}\int_{\partial S} K_1.\eeq
Here the curvature is taken with respect to $S$, so it will be
multiplied by $-1$ on the interior portions of the boundary.  If
$\partial S$ has no corners and consists of an outer boundary and $N$
inner boundary curves, then $C_2=\frac {1-N} 6$.\\

To describe the constant $C_1$ we define
\beq \psi(\theta)=\frac 1 {24} \left(\frac \pi \theta - \frac \theta
\pi\right).\eeq
If we impose Neumann or Dirichlet conditions throughout, then we take
\beq C_1=\sum \psi(\theta_j)\eeq
where the sum is over all corner points: For mixed boundary condition
we subdivide the boundary at corner points, and impose Neumann
conditions on the arcs between some corner points, and Dirichlet
conditions on the remaining arcs.  We then sort the corner points into
those $\{\theta_j'\}$ where the same boundary condition is imposed on
both incident arcs, and $\{\theta_j''\}$ where different boundary
conditions are imposed.  Then we take
\beq C_1=\sum\psi(\theta_j')+\sum(\psi(2\theta_j'')-\psi(\theta_j'')).\eeq

We may also allow the surface to have a finite number of cone point
singularities with cone angles $\{\alpha_j\},\,0<\alpha_j<2\pi$.  In
that case we add to $C_1$ the value
\beq \sum 2\psi_j\left(\frac {\alpha_j} 2 \right).\eeq

\begin{conjecture} (a) Suppose $S$ is flat.  Then there exists a
uniformly almost periodic function $g$ of mean value zero, such that
\beq A(t)=g(t^{1/2})t^{-1/4}+O(t^{-1/2}), \eeq
and more generally there exists a sequence $\{g_j\}$ of uniformly
almost periodic functions, with $g_1=g$, such that for all $N$
\beq A(t)=\sum_{j=1}^N g_j(t^{1/2})t^{-j/4}+O(t^{-(N+1)/4})
\eeq
(b) Suppose $S$ has constant positive curvature.  Then there exists a
uniformly almost periodic function $g$ of mean value zero, such that
\beq A(t)=g(\sqrt{t+\frac 1 4})+O(t^{-1/2}),
\eeq
and more generally there exists a sequence $\{g_j\}$ of uniformly
almost periodic functions, with $g_1=g$, such that for all $N$
\beq A(t)=\sum_{j=1}^N g_j(\sqrt{t+\frac 1 4})t^{-(j-1)/2}+O(t^{-N/2})
\eeq
\end{conjecture}
In this paper we present a number of examples that support the
conjecture.  Of course this is exactly backwards, since the examples
were worked out first and the conjecture was concocted to agree with
the examples.  All of the examples are sufaces for which it is
possible to compute the spectrum of the Laplacian exactly in terms of
trigonometric polynomials for flat surfaces and spherical harmonics
for positively curved surfaces.  In particular, they are highly
symmetric.  So they provide only weak evidence for the conjecture, and
the methods of this paper do not provide a pathway to attack the
conjecture.  Nevertheless, the conjecture seems interesting enough
that it is worth submitting to the mathematical community to stimulate
further work.  We should also point out that although all the almost
periodic functions in our examples of positively curved surfaces are
in fact periodic, we have not made periodicity part of conjecture
(b).\\

For surefaces without boundary the following relationship appears to hold: the set of frequencies in the trigonometric expansions of the almost periodic functions coincides with the set of lengths of closed geodesics on the surface.  At present we have no explanation for why this should be the case.\\

There should be an analogous conjecture for surfaces of constant
negative curvature.  The expression for the refined asymptotics
$\widetilde N(t)$ should be the same, but it is not clear what should
replace (1.11) and (1.13).  There is a vast literature on the spectrum
of the Laplacian on compact Riemann surfaces, or even surfaces of
finite volume, and the relationship with the lengths of closed
geodesics, going back to the Selberg trace formula (see [Bu], [Sa] for
example).  Although the spectral projection operator on the hyperbolic
plane is known explicitly ([Ta], [S]), there is apparently no known
exact computation on any quotient space.  Since we have no examples to
compute, we will not venture a conjecture.\\

There is also a vast literature on the asymptotics of the trace of the
heat kernel ([G]).  The trace of the heat kernel is given as a
smoothing of the eigenvalue counting function
\beq h(t)=\sum e^{t\lambda_j}=t\int^\infty_0 N(s) e^{-st}ds,
\eeq
so asymptotics of $N(t)$ translate immediately into asymptotics of the
heat kernel trace.  If we substitute $N(s)-\widetilde N(s)=\frac d
{ds} (sA(s))$ in (1.15) and integrate by parts we obtain
\beq h(t)-\widetilde h(t)=t^2\int^\infty_0 A(s)se^{-st}\,ds\qquad \text{where}
\eeq
\beq \widetilde h(t)=t\int^\infty_0\widetilde N(s)e^{-st}\, ds,
\eeq
so asymptotics for $A(t)$ also translate into asymptotics of $h(t)$.
To go in the reverse direction requires the application of a Tauberian
Theorem, and cannot reveal the refined asymptotic statements that we
are interested in.  However, we can check that the refined asymptotics
$\widetilde N(t)$ that we use are consistent with the known
asymptotics for $h(t)$ ([BS], [K], [G]).  Our average function $A(t)$
is a much cruder smoothing of $N(t)$ than the trace of the heat
kernel.  This is a double edged sword.  On the one hand, we can't
expect as nice behavior.  On the other hand, the rougher behavior
reveals some interesting new features; for example, the almost
periodic functions.\\

If the boundary is smooth, we observed that the contribution $C_2$ to
the constant term is given by ``topological data,'' namely $C_2=\frac
{1-g} 6$ where $g$ is the genus of the surface.  The analogous
statement for the trace of the heat kernel is the grand finale to the
famous paper of Mark Kac [K].  It is interesting to note that Kac's
argument is based on approximating the smooth boundary by a polynomial
curve, and he is able to say something about the polynomial case in
terms of some frightful integrals, but only in the limit does the
result become comprehensible.  (An explicit formula for the polynomial
case is given in [BS].)\\

In our approach, the polygonal case is quite explicit, but we may also
check that there is continuity in approximating the smooth boundary by
a polygonal boundary.  To simplify the discussion, assume that the
surface is a simply connected convex planar domain and has a smooth
outer boundary $\partial S$.  Subdivide the boundary at $n$ points
$\{x_j\}$, and connect them by line segments to form a polygon $P$.
Then the constant term in $\widetilde N$ for the polygonal domain is
entirely given by $C_1$, namely (1.8) where $\theta_j$ is the angle at
$x_j$.  As $n$ gets large the values of $\theta_j$ approach $\pi$ from
below, and
\beq 24\psi(\theta_j)= 2(1-\frac {\theta_j} \pi) + O((1-\frac
{\theta_j} \pi)^2).
\eeq
A simple geometric argument shows
\beq
\sum \theta_j=(n-2)\pi
\eeq
and so
\beq C_1=\sum\psi(\theta_j)=\frac 1 {24} (2n-\frac 2 \pi
\sum\theta_j)+R=\frac 1 6 + R
\eeq
where the remainder $R$ is $O(\sum(1-\frac{\theta_j} \pi)^2)$ and
tends to zero in the limit.  Thus the constant term for the polygonal
approximations approximates the constant term $(C_2)$ for the smooth
domain.  It is straightforward to localize this argument to polygonal
approximations to portions of the boundary of any surface.\\

We want to emphasize that it is important to average the error $D(t)$
in order to get the asymptotic behavior.  In all our examples the
function $D(t)$ is unbounded.  It is the fact that it is positive and
negative that allows cancellation to produce bounded behavior for
$A(t)$.  Compared with the values of $D(t)$, the constant term in the
asymptotics $\widetilde N(t)$ is indeed ``exceeding small.''  The fact
that it nevertheless shows up in the average seems truly remarkable.\\

Similar results should be valid in higher dimensions, but they will
require a different type of average.  We leave this to the future.\\

We now outline the examples that take up the rest of this paper.
There are two basic examples, the torus discussed in [JS] and the
sphere, discussed in section 4.  In all the other examples, the
eigenfunctions on the surface may be extended to eigenfunctions on a
torus or a sphere, so the spectrum on the surface is a subset of the
spectrum for one of our basic examples.  Our task is then to identify
exactly the subset, and show how the conjecture for the basic example
yields the conjecture for the surface.  This requires only elementary
reasoning, but the arguments are a bit technical.  Of course we have
to be very careful, since the value of the constant depends on getting
exact statements.  It may seem that we are working out a lot of
examples using very similar arguments.  However, we found that we
needed all the examples to help formulate and confirm the conjecture.
We have tried to present enough detail so that the reader can verify
the correctness of the result, without excessively repeating the
framework of the reasoning.  We have used what we hope is self
explanatory notation, but it changes from example to example.  So, for
example, we write $N(t)$ for the counting function for the surface
under discussion (or $N_N(t)$ or $N_D(t)$ if there is a boundary with
Neumann or Dirichlet boundary conditions).  If we have to recall a
counting function from a previous example we will write $N_T(t)$ for a
torus $T$, etc.\\

In section 2 we discuss examples that are polygons in the plane with
either Neumann or Dirichlet boundary conditions throughout.  These
examples are arbitrary rectangles, and certain special triangles:
right isosceles, equilateral, and $30^\circ-60^\circ-90^\circ$.\\

In section 3 we return to the same surfaces, but deal with mixed
boundary conditions, Neumann on part of the boundary and Dirichlet on
the remainder.  We are able to handle all possibilities for the
rectangle and right isosceles triangle, but none for the equilateral
triangle and only some for the $30^\circ-60^\circ-90^\circ$ triangle.
In this section we also consider an arbitrary cylinder and M\"obius
band.\\

In section 4 we consider the sphere, the hemisphere and the projective
sphere.  In section 5 we consider lunes and half-lunes.  These
examples are very useful because they provide a wider variety of
corner angles than the previous examples.\\

In section 6 we discuss surfaces with point singularities.  These
include the flat projective plane discussed in [JS], the surface of a
regular tetrahedron discussed in [GKS], half tetrahedra, and
glued-lunes.\\

In section 7 we re-examine some of our examples of surfaces that have
a finite group $G$ of isometries.  The question is how the spectrum
sorts into the eigenfunctions with prescribed symmetry, given by the
irreducible representations $\{\pi_j\}$ of $G$.  A heuristic suggested
in [S] is that the proportions $N_j(t)/N(t)$, where $N_t(t)$ is the
counting function for eigenfunctions of $\pi_j$ symmetry types is
asymptotically $(\dim \pi_j)^2/\#G$.  Here we work out refined
asymptotics $\widetilde N_j(t)$ for $N_j(t)$, by reducing the
computation of $N_j(t)$ to previously worked out examples for a
fundamental domain of the $G$ action.  The leading term $At$ is as
predicted, but the next term $Bt^{1/2}$ may be positive, negative, or
zero, showing that the individual representations are somewhat
overrepresented or underepresented in the spectrum.  Because we only
have a few examples, we are not able to formulate a conjecture for the
behavior in the general case.  This is another interesting open
problem for the future.\\

\section{Polygons}
Consider the rectangle $R$ of side length, $a,b$, and let $T$ denote
the torus of side length $2a,2b$.  Any Neumann eigenfunction on $R$
extends by even reflection to an eigenfunction on $T$, and similarly a
Dirichlet eigenfunction extends by odd reflection.  In either case we
obtain roughly a quarter of the eigenfunction on $T$, but in fact we
can give a precise formula relating the counting functions $N_N$ and
$N_D$ for $R$ with $N_T$ for $T$.\\

The eigenfunctions on $T$ have the form
\beq e(j,k)=e^{\pi i(\frac j a x+\frac k b y)},\qquad j,k\in\mathbb Z,\eeq
with eigenvalue $\frac {\pi^2j^2} {a^2}+\frac {\pi^2k^2}{b^2}$.  The
Neumann eigenfunctions on $R$ have the form
\beq c(j,k)=\cos\pi \frac j a x\cos\pi \frac k b y,\qquad j\ge0,\,k\ge0,\eeq
while the Dirichlet eigenfunctions have the form
\beq s(j,k)=\sin\pi\frac j a x\sin \pi\frac k b y,\qquad j>0,\,k>0,\eeq
with the same eigenvalue.\\
\begin{lemma}
We have \beq N_N(t)=\frac 1 4 N_T(t)+\frac 1 2 \left[\frac {at^{1/2}}
\pi \right] + \frac 1 2 \left[\frac {bt^{1/2}} \pi \right] +\frac 3
4\eeq
\beq N_D(t)=\frac 1 4 N_T(t)-\frac 1 2 \left[\frac {at^{1/2}} \pi
\right] - \frac 1 2 \left[\frac {bt^{1/2}} \pi \right] -\frac 1 4\eeq
\end{lemma}
\begin{proof}
In the generic case $j>0$ and $k>0$, we have four eigenfunctions
$e(\pm j,\pm k)$ contributing to $N_T$ for one eigenfunction $c(j,k)$
or $s(j,k)$ contributing to $N_N$ or $N_D$, giving rise to the $\frac
1 4 N_T(t)$ terms in (2.4) and (2.5).  We then have to correct for the
nongeneric cases.  Note that $e(0,0)$ contributes once to $N_T$ and
$N_N$ but not to $N_D$, so this gives rise to the constant terms.
When $k=0$ and $j>0$, we have $e(\pm j, 0)$ for $\frac {\pi^2j^2}
{a^2}\le t$ contributing to $N_T(t)$ and $c(j,0)$ contributing to
$N_N(t)$ but not $N_D(t)$, so this gives rise to the term $\pm\frac 1
2 \left[\frac {at^{1/2}} \pi\right]$ in (2.4) and (2.5).  Similarly
the case $k>0$ and $j=0$ gives rise to the term $\pm\frac 1 2
\left[\frac {bt^{1/2}}\pi\right]$.
\end{proof}

Note that the function $[x]$ is on average $x-\frac 1 2$, so we may
rewrite (2.4) and (2.5) as
\beq N_N(t)=\frac 1 4 N_T(t)+\frac 1 2 \left ( \left[\frac { a
t^{1/2}} \pi \right] + \frac 1 2 \right ) + \frac 1 2 \left (
\left[\frac { b t^{1/2}} \pi \right] + \frac 1 2 \right ) + \frac 1 4
\eeq
\beq N_D(t)=\frac 1 4 N_T(t)-\frac 1 2 \left ( \left[\frac { a
t^{1/2}} \pi \right] + \frac 1 2 \right ) - \frac 1 2 \left (
\left[\frac { b t^{1/2}} \pi \right] + \frac 1 2 \right ) + \frac 1 4
\eeq
Now the constant term is the same in both equations.  We define the
refined asymptotics
\beq \widetilde N_N(t)=\frac {ab} {4\pi} t + \frac {2a+2b}{4\pi}
t^{1/2}+\frac 1 4 \eeq
\beq \widetilde N_D(t)=\frac {ab} {4\pi} t - \frac {2a+2b} {4\pi}
t^{1/2} + \frac 1 4.\eeq
Note that $ab$ is the area of $R$ and $2a+2b$ is the length of the
perimeter of $R$ as predicted, and
$\frac 1 4=\frac 1 {16} + \frac 1 {16} + \frac 1 {16} + \frac 1 {16}$
as predicted.  As usual we define the discrepancy $D(t)$
and average $A(t)$ in each of the three cases.
\begin{theorem} Both $A_N$ and $A_D$ satisfy
\beq A(t)=g(t^{1/2})t^{-1/4}+ O(t^{-1/2})\quad \text{as $t\rightarrow
\infty$}\eeq
where $g$ is an almost periodic function of mean value zero.\end{theorem}
\begin{proof} We have\\
\beq \left\{\begin{array}{c}D_N(t)\\D_D(t)\end{array}\right\}=\pm\frac
1 2 \left( \left[\frac {at^{1/2}} \pi \right]+\frac 1 2 - \frac
{at^{1/2}} \pi\right )\pm\frac 1 2 \left( \left[\frac {bt^{1/2}} \pi
\right]+\frac 1 2 - \frac {bt^{1/2}} \pi\right )+\frac 1 4 D_T(t).
\eeq\\
Since (2.10) holds for $A_T$ by Theorem 4 of [JS], it suffices to show
that an estimate like (2.10) holds for the function $[ct^{1/2}]+\frac
1 2 - ct^{1/2}$.  For simplicity take $c=1$.  We need to estimate
\beq
\begin{split}
\frac 1 t \int^t_0([s^{1/2}]+\frac 1 2 - s^{1/2})\,ds=\frac 2 t
\int^{\sqrt t}_0([r]+\frac 1 2 - r)r\,dr \\
=\frac 2 t \sum^{[t^{1/2}]}_{k=1}\int^k_{k-1} (k-\frac 1 2 -
r)\,r\,dr\,+\frac 2 t\int^{t^{1/2}}_{[t^{1/2}]}([t^{1/2}]+\frac 1 2 -
r)\,r\,dr.
\end{split}
\eeq
Now $\int^k_{k-1}(k-\frac 1 2 -r)\,r\,dr=-\frac 1 4$, so the first
term in (2.12) is exactly $\frac {-[t^{1/2}]} {2t}=O(t^{-1/2})$.  For
the second term we note that the integrand is $O(t^{1/2})$ and the
interval of integration has length at most 1, so again the
contribution is $O(t^{-1/2})$.\end{proof}

The almost periodic function $g$ is the same for $A_N$ and $A_D$, and
aside from the factor $1/4$ it is given explicitly in [JS].  The error
estimate $O(t^{-1/2})$ in (2.10) is somewhat worse than the estimate
$O(t^{-3/4})$ for $A_T$ given in [JS].  We may also regard (2.10) as
the first term in an asymptotic expansion with $O(t^{-1/2})$ replaced
by sums of $g_k(t^{1/2})t^{-k/4}$.  For odd values of $k$ the $g_k$
are almost periodic functions as given in [JS] arising from $\frac 1 2
A_T$.  The expression (2.12) may be written as
$g_2(t^{1/2})t^{-1/2}+g_4(t^{1/2})t^{-1}$ where $g_2$ and $g_4$ are
periodic of period 1.  This is easily seen because $h(x)=x-[x]$ is
such a periodic function, and (2.12) is a polynomial in $h(t^{1/2})$
and $t^{1/2}$ divided by $t$.\\

Next we consider a right isosceles triangle that is half of a square
of side length $a$.  Then Neumann and Dirichlet eigenfunctions extend
by even and odd reflection to eigenfunctions of the same type.
Continuing the same notation as before with $b=a$, the Neumann
eigenfunctions are
\beq
c(j,k)+c(k,j)\qquad \text{for $0\le j\le k$}
\eeq
and the Dirichlet eigenfunctions are
\beq
s(j,k)-s(k,j)\qquad \text{for $0<j<k$.}
\eeq

\begin{lemma} For the right isosceles triangle we have
\beq N_N(t)=\frac 1 8 N_T(t)+\frac 1 2 \left(\left[\frac {at^{1/2}}
{\pi}\right]+\frac 1 2\right)+\frac 1 2 \left(\left[\frac
{at^{1/2}}{\sqrt 2 \pi}\right]+\frac 1 2\right)+\frac 3 8
\eeq
\beq N_D(t)=\frac 1 8 N_T(t)-\frac 1 2 \left(\left[\frac {at^{1/2}}
{\pi}\right]+\frac 1 2\right)-\frac 1 2 \left(\left[\frac
{at^{1/2}}{\sqrt 2 \pi}\right]+\frac 1 2\right)+\frac 3 8
\eeq\end{lemma}
\begin{proof}
In the generic case $0<j<k$ there are eight eigenfunctions $e(\pm
j,\pm k)$ and $e(\pm k,\pm j)$ contributing to $N_T$ for one
eigenfunction (2.13) or (2.14) contributing to $N_N$ or $N_D$.  The
case $(j,k)=(0,0)$ contributes the constant.  The case $k>j=0$
contributes $e(\pm k,0)$ and $e(0,\pm k)$ to $N_T$, but only
$c(k,0)+c(0,k)$ to $N_N$ and nothing to $N_D$, leading to the term
$\frac 1 2 \left[\frac {at^{1/2}} \pi\right]$ in (2.15) and (2.16).
When $j=k>0$ we have the eigenvalue $\frac {2\pi^2k^2}{a^2}$, with
$e(\pm k,\pm k)$ contributing to $N_T$, only $c(k,k)$ contributing to
$N_D$, and no contribution to $N_N$.  This leads to the term $\frac 1
2\left[\frac{at^{1/2}}{\sqrt 2\pi}\right]$ in (2.15) and
(2.16).\end{proof}

Thus we define the refined asymptotics
\beq \widetilde N_N(t)=\frac {a^2} {8\pi}t +\left(\frac {2+\sqrt
2}{4\pi}\right)at^{1/2}+\frac 3 8,\eeq
\beq \widetilde N_D(t)=\frac {a^2} {8\pi}t -\left(\frac {2+\sqrt
2}{4\pi}\right)at^{1/2}+\frac 3 8,\eeq

\begin{theorem} For the right isosceles triangle we have the estimate
(2.10) for $A_N$ and $A_D$
\end{theorem}
\begin{proof} Same as for Theorem 2.2.
\end{proof}

Note that $\frac 3 8=\frac 1 {16} + \frac 5 {32} + \frac 5 {32}$ as
predicted for the angles $\frac \pi 2$, $\frac \pi 4$, $\frac \pi
4$.\\

Next we consider the equilateral triangle.  For simplicity we assume
the side length is 1.  Six equilateral triangles tile a regular
hexagon, (see figure 2.1), and Dirichlet and Neumann eigenfunctions
extended by even or odd reflection extend to periodic functions on the
plane with respect to the lattice associated to the hexagonal torus
$T$.

\begin{figure}[h]
\centering 
\includegraphics[scale=.6]{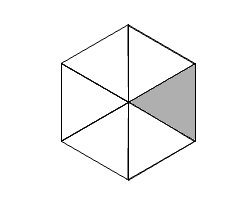}\caption{}
\end{figure}
\noindent We may choose a basis $(\sqrt 3, 0)$ and $(\frac {\sqrt 3} 2,\frac 3
2)$ for the lattice $\mathscr L$, and a basis $u=(\frac 1 {\sqrt 3},
\frac 1 3)$ and $v=(0,\frac 2 3)$ for the dual lattice $\mathscr L^*$.
 Then the torus eigenfunctions are of the form
\beq \widetilde e(k,j)=e^{2\pi(ku+jv)\cdot x}\qquad k,j\in\mathbb Z\eeq
with eigenvalue $\left(\frac {4\pi} 3\right)^2(k^2+j^2+kj).$\\

To describe the Neumann and Dirichlet eigenfunctions on the triangle
it is convenient to think of the dual lattice $\mathscr L^*$ as made
up of the origin surrounded by concentric hexagons.  In Figure 2.2 we
show one such hexagon with a generic choice $(k>j>0)$ of twelve points
associated to the same eigenvalue, together with three reflection axes
of the triangle sides.

\begin{figure}[h!]
\centering
\includegraphics[scale=.6]{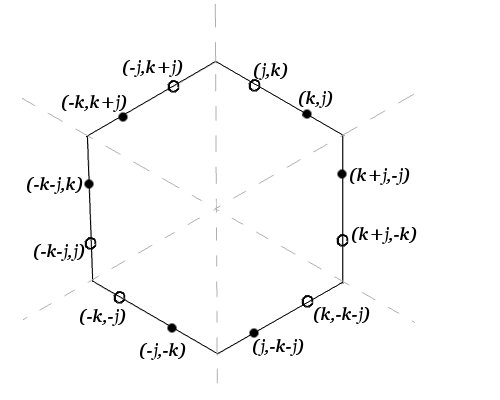}\caption{}
\end{figure}
\noindent A Neumann eigenfunction must be symmetric with respect to the three
reflections, so there are two eigenfunctions associated to these
lattice points, namely
\beq \begin{split}\widetilde e(k,j)+\widetilde e(-k,k+j)+\widetilde
e(-k-j,k)+\widetilde e(-j,-k)\left.\,\,\,\,\right.\\+\widetilde
e(j,-k-j)+\widetilde e(k+j,-j)\end{split}\eeq
(dark points) and the same with $j$ and $k$ interchanged (open
points).  Similarly a Dirichlet eigenfunction must be skew-symmetric,
so again we find two,
\beq \begin{split}\widetilde e(k,j)-\widetilde e(-k,k+j)+\widetilde
e(-k-j,k)-\widetilde e(-j,-k)\left.\,\,\,\,\right.\\+\widetilde
e(j,-k-j)-\widetilde e(k+j,-j)\end{split}\eeq
and the same with $j$ and $k$ interchanged.  For the nongeneric cases
we have $(0,0)$ contributing to $N_T$ and $N_N$ but not $N_D$, $(k,k)$
contributing six terms to $N_T$ and one each to $N_N$ and $N_D$, and
$(k,0)$ contributing six terms to $N_T$ and two terms to $N_N$, namely
$\widetilde e(k,0)+\widetilde e(-k,k)+\widetilde e(0,-k)$ and
$\widetilde e(-k,0)+\widetilde e(k,-k)+\widetilde e(0,k)$ and nothing
to $N_D$.  This yields
\beq N_N(t)=\frac 1 6 N_T(t)+\left(\left[\frac 3 {4\pi}
t^{1/2}\right]+\frac 1 2\right) + \frac 1 3,
\eeq
\beq N_D(t)=\frac 1 6 N_T(t)-\left(\left[\frac 3 {4\pi}
t^{1/2}\right]+\frac 1 2\right) + \frac 1 3.\eeq
Thus we define the refined asymptotics
\beq \widetilde N_N(t)=\frac 1 {4\pi}\frac {\sqrt 3} 4 t + \frac 3
{4\pi} t^{1/2}+\frac 1 3,
\eeq
\beq \widetilde N_D(t)=\frac 1 {4\pi}\frac {\sqrt 3} 4 t - \frac 3
{4\pi} t^{1/2}+\frac 1 3,
\eeq
and the analog of Theorem 2.2 is valid.  Note that $\frac 1 3=\frac 1
9 +\frac 1 9 +\frac 1 9$ as predicted.\\

Finally, we consider the $30^\circ-60^\circ-90^\circ$ triangle that is
half of the equilateral triangle.  For Neumann eigenfunctions that
means we take Neumann eigenfunctions on the equilateral triangle that
are symmetric with respect to reflection in the $x$-axis, and
similarly for Dirichlet eigenfunctions with skew-symmetry.  For a
generic choice $k>j>0$ we get exactly one eigenfunction of each type
by taking the sum of the two eigenfunctions of the form (2.20) or
(2.21), with $j$ and $k$ interchanged.   For the nongeneric cases we
have $(0,0)$ contributing to $N_T$ and $N_N$ but not $N_D$, $(k,k)$
contributing six terms to $N_T$, one term to $N_N$ and nothing to
$N_D$, and $(k,0)$ contributing six terms to $N_T$, one term to $N_N$
and nothing to $N_D$.  Thus
\beq \begin{split} N_N(t)=\frac 1 {12} N_T(t)+\frac 1 2
\left(\left[\frac {\sqrt 3} {4\pi} t^{1/2}\right]+\frac 1 2\right)
+\frac 1 2 \left (\left [ \frac 3 {4\pi} t^{1/2}\right]+\frac 1 2
\right) + \frac 5 {12},
\end{split}\eeq
\beq \begin{split} N_D(t)=\frac 1 {12} N_T(t)-\frac 1 2
\left(\left[\frac {\sqrt 3} {4\pi} t^{1/2}\right]+\frac 1 2\right)
-\frac 1 2 \left (\left [ \frac 3 {4\pi} t^{1/2}\right]+\frac 1 2
\right) + \frac 5 {12}.
\end{split}\eeq
Thus the analog of Theorem 2.2 holds for
\beq \widetilde N_N(t)=\frac 1 {4\pi} \frac {\sqrt 3} 8 t +
\frac{3+\sqrt 3} {8\pi} t^{1/2}+\frac 5 {12},\eeq
\beq \widetilde N_D(t)=\frac 1 {4\pi} \frac {\sqrt 3} 8 t -
\frac{3+\sqrt 3} {8\pi} t^{1/2}+\frac 5 {12}.\eeq
Note that $\frac {3+\sqrt 3} 2$ is the length of the perimeter of the
triangle, and $\frac 5 {12}= \frac 1 {16}+\frac 1 {9}+\frac {35}
{144}$ as predicted.

\section{Mixed Boundary Conditions}
In this section we look at examples of polygons with Neumann
conditions on part of the boundary and Dirichlet conditions on the
rest of the boundary.\\

We begin with the simplest case: a rectangle $R$ with Neumann
conditions on facing edges of length $a$ and Dirichlet conditions on
the facing edges of length $b$.  Just as in the pure Neumann and
Dirichlet example in section 2 we may extend eigenfunctions to the
same torus $T$, and so the eigenfunctions have the form
\beq f(j,k)=\cos\pi\frac j a x\sin\pi\frac k b y\qquad \text{for
$j\ge0$ and $k> 0$,}
\eeq

with eigenvalue $\frac {\pi^2 j^2} {a^2} + \frac {\pi^2k^2} {b^2}$.
Just as in Lemma 2.1 we have
\beq \begin{split} N(t)=\frac 1 4  N_T(t)+\frac 1 2 \left[\frac
{at^{1/2}} \pi\right]-\frac 1 2 \left[\frac{bt^{1/2}} \pi\right] -
\frac 1 4\left.\,\,\,\,\,\,\right.\\
 =\frac 1 4  N_T(t)+\frac 1 2 \left(\left[\frac {at^{1/2}}
\pi\right]+\frac 1 2\right)-\frac 1 2 \left(\left[\frac{bt^{1/2}}
\pi\right]+\frac 1 2\right) - \frac 1 4,
\end{split}\eeq

and so we define the refined asymptotics
\beq \widetilde N(t)=\frac {ab} {4\pi}t+\frac {2a-2b}{4\pi}t^{1/2}-\frac 1 4.
\eeq

\noindent Then the analog of Theorem 2.2 holds.  Note that $2a-2b$
gives the difference of the lengths of the portions of the perimeter
where Neumann and Dirichlet boundary conditions hold, and the constant
is $-\frac 1 4$ because each of the four vertices has mixed boundary
conditions on the incident edges, with $\psi(\pi)-\psi(\frac \pi
2)=-\frac 1 {16}$.\\

Next we consider the case where the same type of boundary conditions
hold on one pair of opposite edges (say the ones of length $b$), while
for the other pair we have Neumann on one side and Dirichlet on the
other.  We call these the NM (Neumann/mixed) and DM (Dirichlet/mixed)
cases.  Here we need a larger torus $T'$ of size $2a\!\!\times\!\!4b$
on which to extend the eigenfunctions.  The eigenfunctions then have
the form
\beq f(j,k)=\left\{\begin{array}{rl} \cos\pi\frac j a
x\cos\pi\left(\frac {k+\frac 1 2} b\right)y & j\ge 0,k\ge 0\quad\text{
 (NM)}\\
\sin\pi\frac j a x\cos\pi\left(\frac {k+\frac 1 2} b\right)y &
j>0,k\ge 0\quad\text{  (DM)}
\end{array}\right.
\eeq
with eigenvalue
\beq \frac {\pi^2j^2} {a^2} + \frac {\pi^2 (2k+1)^2} {4b^2}.
\eeq

Now we observe that $N_{T'}(t)$ also has eigenvalues $\frac {\pi^2j^2}
{a^2} + \frac {\pi^2 (2k)^2} {4b^2}$, so $N_{T'}(t)-N_T(t)$ counts all
eigenvalues of the form (3.5) with multiplicity $4$ for the generic
case $j>0$ and 2 for the case $j=0$.  Thus
\beq N(t)=\frac 1 4\left(N_{T'}(t)-N_T(t)\right)\pm\frac 1 2
\left[\frac b \pi t^{1/2}+\frac 1 2\right]
\eeq
($+$ for NM and $-$ for DM cases),
since the condition $\frac {\pi^2 (2k+1)^2} {4b^2}\le t$ means $0\le
k\le \left[\frac b \pi t^{1/2}-\frac 1 2\right]$.  Thus we choose the
refined asymptotics
\beq \left\{\begin{split}\widetilde N_{NM}(t)=\frac {ab} {4\pi}t+\frac
{2b} {4\pi} t^{1/2}\\
\widetilde N_{DM}(t)=\frac {ab} {4\pi}t-\frac {2b} {4\pi} t^{1/2},
\end{split}\right.\eeq
and the analog of Theorem 2.2 holds.  Indeed $\widetilde
N_{T'}(t)-\widetilde N_T(t)=\frac {8ab} {4\pi} t-\frac {4ab} {4\pi}t$,
and
$\left[\frac b \pi t^{1/2}+\frac 1 2\right]-\frac b \pi
t^{1/2}=O(t^{-1/2})$ as in the proof of Theorem 2.2.  The explanation
for the coefficient of $t^{1/2}$ in (3.7) is that the sides of length
$a$ cancel because they have mixed boundary conditions, while the
sides of length $b$ add because they have like boundary conditions.
There is no constant term because there are two vertices with mixed
boundary conditions and two with like boundary conditions on their
incident edges.\\

The last case of mixed boundary conditions on $R$ has mixed conditions
on both pairs of opposite edges (MM).  In this case we need a still
larger torus $T''$ of size $4a\!\times\!4b$ on which to extend the
eigenfunctions.  Then the eigenfunctions have the form
\beq
f(j,k)=\cos\pi\left(\frac {j+\frac 1 2} a\right)x\cos\pi\left(\frac
{k+\frac 1 2} b\right)y\qquad\text{for $j\ge 0, k\ge0$}
\eeq with eigenvalue
\beq \frac {\pi^2 (2j+1)^2} {4a^2}+\frac {\pi^2 (2k+1)^2} {4b^2}.
\eeq

If we denote by $T'_1$ and $T'_2$ the tori of sizes $2a\!\times \!4b$
and $4b\!\times \!2a$, then $N_{T''}-N_{T'_1}-N_{T'_2}+N_T$ counts all
eigenvalues of the form (3.9) with multiplicity 4, and all eigenvalues
are generic.  So
\beq N_{MM}(t)=\frac 1 4 \left( N_{T''}(t)-N_{T'_1}(t)-N_{T'_2}(t)+N_T(t)\right)
\eeq
Thus we define the refined asymptotics
\beq \widetilde N_{MM}(t)=\frac {ab} {4\pi}t,
\eeq
and the analog of Theorem 2.2 holds because $\frac 1 4 \left(
\widetilde N_{T''}(t)-\widetilde N_{T'_1}(t)-\widetilde
N_{T'_2}(t)+\widetilde N_T(t)\right)=\widetilde N_{MM}(t)$.  Here
there is no $t^{1/2}$ term in (3.11) because of cancellation of
opposite sides, and no constant term because of cancellation of vertex
pairs.\\

Next we look at flat cylinders $C$ obtained from the rectangles $R$ by
identifying the opposite edges of side length $b$.  We impose either
Neumann (N), Dirichlet (D) or mixed (M) boundary conditions on the
other pair of opposite edges.  Then reflection produces eigenfunctions
on the torus $T'_3$ of size $a\!\times\!2b$ (N or D) or $T'_4$ of size
$a\!\times\!4b$ (M).  The eigenfunctions are given by
\beq f(j,k)=\left\{\begin{array}{lll} e^{2\pi i\frac j a
x}\cos\pi\frac k b y&\text{ for $j\in\mathbb Z, k\ge 0$}& (\text{N})\\
e^{2\pi i\frac j a x}\sin\pi\frac k b y&\text{ for $j\in\mathbb Z, k>
0$}& (\text{D})\\
e^{2\pi i\frac j a x}\cos\pi\left(\frac {k+\frac 1 2} b\right)
y&\text{ for $j\in\mathbb Z, k\ge 0$}& (\text{M})\\
\end{array}\right.
\eeq
with eigenvalue
\beq \left\{\begin{array}{ll}\frac {4\pi^2j^2} {a^2} +\frac {\pi^2k^2}
{b^2}& (\text{N or D})\vspace{3mm}\\
\frac {4\pi^2j^2} {a^2} +\frac {\pi^2(2k+1)^2} {4b^2}& (\text{M}).\\
\end{array}\right.
\eeq
Reasoning as before we find
\beq \left\{\begin{array}{l}
N_N(t)=\frac 1 2 N_{T'_3}(t)+\left[\frac {at^{1/2}}
{2\pi}\right]+\frac 1 2\vspace{3mm}\\
N_D(t)=\frac 1 2 N_{T'_3}(t)-\left[\frac {at^{1/2}}
{2\pi}\right]-\frac 1 2\vspace{3mm}\\
N_M(t)=\frac 1 2 \left(N_{T'_4}(t)-N_{T'_3}(t)\right).
\end{array}\right.
\eeq
We define the refined asymptotics
\beq \left\{\begin{array}{l}
\widetilde N_N(t)=\frac {ab} {4\pi}t+\frac {2a} {4\pi}t^{1/2}\vspace{2mm}\\
\widetilde N_D(t)=\frac {ab} {4\pi}t-\frac {2a} {4\pi}t^{1/2}\vspace{2mm}\\
\widetilde N_M(t)=\frac {ab} {4\pi}t,
\end{array}\right.
\eeq
and the analog of Theorem 2.2 holds.\\

Next we consider a M\"obius band obtained from the rectangle $R$ by
identifying the sides of length $a$ with reversed orientation, with
either Neumann or Dirichlet boundary condition on the boundary curve
of length $2b$.  We can extend eigenfunctions to the torus $T$, and we
find they have the form

\beq f(j,k)=\left\{\begin{array}{lll} e^{\pi i\frac j a x}\cos\frac
{\pi k} b y&\text{ if $k\ge 0$ and $j+k$ is even}& (\text{N})\\
e^{\pi i\frac j a x}\sin\frac {\pi k} b y&\text{ if $k> 0$ and $j+k$
is odd}& (\text{D})\\
\end{array}\right.
\eeq
with eigenvalue $\frac {\pi^2j^2}{a^2}+\frac {\pi^2k^2}{b^2}$.
Consider the lattice counting function
\beq \left\{\begin{array}{l} N_e(t)=\#\{(j,k)\,:\,j+k\text{ is even
and }\frac {\pi^2j^2}{a^2}+\frac{\pi^2k^2}{b^2}\le t\}\vspace{2mm}\\
N_o(t)=\#\{(j,k)\,:\,j+k\text{ is odd and }\frac
{\pi^2j^2}{a^2}+\frac{\pi^2k^2}{b^2}\le t\}.
\end{array}\right.
\eeq
Reasoning as before we find
\beq \left\{\begin{array}{l}N_N(t)=\frac 1 2
N_e(t)+\left[\frac{at^{1/2}}{2\pi}\right]+\frac 1 2\vspace{2mm}\\
N_D(t)=\frac 1 2 N_o(t)-\left[\frac{at^{1/2}}{2\pi}+\frac 1 2\right].
\end{array}\right.
\eeq
Now $N_e(t)$ is equal to the counting function corresponding to a
torus $T'_5$ of area $2ab$, and $N_o(t)=N_T(t)-N_{T'_5}(t)$, so both
$\frac 1 2 N_e(t)$ and $\frac 1 2 N_o(t)$ are covered by the results
of [JS].  Thus we define the refined asymptotics
\beq \left\{\begin{array}{l}\widetilde N_N(t)=\frac {ab} {4\pi}t+\frac
{2a} {4\pi} t^{1/2}\\
\widetilde N_D(t)=\frac {ab} {4\pi}t-\frac {2a} {4\pi} t^{1/2},
\end{array}\right.
\eeq
with the analog of Theorem 2.2 holding.  Note that the refined
asymptotics of the M\"obius band is identical to that of the cylinder,
although the eigenvalue counting functions are not equal, and the
almost-periodic function $g$ is not the same.\\

Next we consider the isosceles triangle with mixed boundary
conditions.  There are essentially four different cases.  In the first
two we consider the same type of boundary condition on the equal
sides, and the opposite type on the hypotenuse.  We write (ND) for
Neumann on the equal sides, and (DN) for Dirichlet on the equal sides.
 In the ND case we take an odd reflection in the hypotenuse to obtain
a pure Neumann condition on the square.  If we compare with Lemma 2.3
where we took an even reflection, we see that the eigenfunctions on
the square with Neumann conditions split into the mixed ND
eigenfunctions and the pure N eigenfunctions, so
\beq N_{ND}(t)+N_N(t)=N_N^{(S)}(t),\eeq
where the right side denotes the counting function for the square.  Similarly
\beq N_{DN}(t)+N_D(t)=N_D^{(S)}(t).\eeq
Thus, $N_{ND}(t)$ is given by the difference of (2.6) (with $b=a$) and
(2.15), so
\beq N_{ND}(t)=\frac 1 8 N_T(t)+\frac 1 2 \left(\left[\frac {at^{1/2}}
\pi\right]+\frac 1 2\right)-\frac 1 2\left(\left[\frac {at^{1/2}}
{\sqrt 2 \pi}\right]+\frac 1 2\right)-\frac 1 8.
\eeq
Similarly $N_{DN}$ is the difference of (2.7) and (2.15), so
\beq N_{DN}(t)=\frac 1 8 N_T(t)-\frac 1 2 \left(\left[\frac {at^{1/2}}
\pi\right]+\frac 1 2\right)+\frac 1 2\left(\left[\frac {at^{1/2}}
{\sqrt 2 \pi}\right]+\frac 1 2\right)-\frac 1 8.
\eeq
Here the constant is $\psi(\frac \pi 2)+2\left(\psi(\frac \pi
2)-\psi(\frac \pi 4)\right)=\frac 3 {16} - 2(\frac 5 {32})=-\frac 1
8.$\\

This leads to the choice
\beq \widetilde N_{ND}(t)=\frac {a^2} {8\pi}t+\left(\frac { 2-\sqrt
2}{4\pi}\right)at^{1/2}-\frac 1 8,
\eeq
\beq \widetilde N_{DN}(t)=\frac {a^2} {8\pi}t+\left(\frac {-2+\sqrt
2}{4\pi}\right)at^{1/2}-\frac 1 8,
\eeq
and the analog of Theorem 2.4 holds.\\

In the last two cases we have mixed boundary conditions on the two
equal sides and Neumann (MN) or Dirichlet (MD) conditions on the
hypotenuse.  By reflecting evenly or oddly in the hypotenuse we end up
in the (MM) case for the square.  As in the earlier cases the (MM)
eqigenfunctions on the square (3.8) yield the (MN) eigenfunctions
$f(j,k)+f(k,j)$ and the (MD) eigenfunctions $f(j,k)-f(k,j)$  $(k\ne
j)$.  In the place of (3.20) and (3.21) we have
\beq N_{MN}(t)+N_{MD}(t)=N_{MM}^{(S)}(t),\eeq
but this does not allow us to immediately compute the two summands.
However, in the generic case $j\ne k$ we obtain one eigenfunction in
each case for two in the square, while in the nongeneric case $j=k$ we
obtain one in the (MN) case and none in the (MD) case for one in the
square.  Thus
\beq N_{MN}(t)=\frac 1 2 N^{(S)}_{MM}(t)+\frac 1 2 \left[\frac {\sqrt
2 a}{2\pi} t^{1/2}+\frac 1 2\right]
\eeq
and
\beq N_{MD}(t)=\frac 1 2 N^{(S)}_{MM}(t)-\frac 1 2 \left[\frac {\sqrt
2 a}{2\pi} t^{1/2}+\frac 1 2\right].
\eeq
This leads to the choice
\beq \widetilde N_{MN}(t)=\frac{a^2} {8\pi}t+\left(\frac {\sqrt 2}
{4\pi}\right)at,
\eeq
\beq \widetilde N_{MD}(t)=\frac{a^2} {8\pi}t-\left(\frac {\sqrt 2}
{4\pi}\right)at,
\eeq
with the analog of Theorem 2.4 holding.  Here the constant is
$\left(\psi(\pi)-\psi(\frac \pi 2)\right)+\psi(\frac\pi
4)+\left(\psi(\frac \pi 2)-\psi(\frac \pi 4)\right)=0.$\\

In our last example we consider two mixed boundary problems on the
$30^\circ-60^\circ-90^\circ$ triangle with the same type of boundary
condition on the hypotenuse and the shortest side, and the opposite
type on the side that cuts the equilateral triangle in half that we
call (ND) and (DN).  In the (ND) case an odd reflection across the cut
side produces a Neumann boundary condition on the equilateral
triangle, so these eigenfunctions together with the pure Neumann
eigenfunctions split up the pure Neumann eigenfunctions on the
equilateral triangle.  Thus $N_{ND}(t)$ is given by the difference of
(2.22) and (2.26), hence
\beq\begin{split}
N_{ND}(t)=\frac 1 6 N_T(t)+\left(\left[\frac 3 {4\pi}
t^{1/2}\right]+\frac 1 2\right)+\frac 1 3\hspace{1cm}\\
-\frac 1 {12} \left(N_T(t)+\frac 1 2 \left(\left[\frac 3
{4\pi}t^{1/2}\right]+\frac 1 2 \right)+\frac 1 2
\left(\left[\frac{\sqrt 3} {4\pi} t^{1/2}\right]+\frac 1
2\right)+\frac 5 {12}\right)\\
=\frac 1 {12} N_T(t)+\frac 1 2 \left(\left[\frac 3 {4\pi}
t^{1/2}\right]+\frac 1 2 \right)-\frac 1 2 \left(\left[\frac {\sqrt3}
{4\pi} t^{1/2}\right]+\frac 1 2 \right)-\frac 1 {12}.
\end{split}
\eeq

Similarly $N_{DN}(t)$ is the difference of (2.23) and (2.27), hence
\beq
N_{DN}(t)=\frac 1 {12} N_T(t)-\frac 1 2 \left(\left[\frac 3 {4\pi}
t^{1/2}\right]+\frac 1 2 \right)+\frac 1 2 \left(\left[\frac {\sqrt 3}
{4\pi} t^{1/2}\right]+\frac 1 2\right)-\frac 1 {12}.
\eeq

This leads to the choice
\beq
\widetilde N_{ND}(t)=\frac 1 {4\pi} \frac {\sqrt 3} 8 t+\frac {3-\sqrt
3} {8\pi} t^{1/2}-\frac 1 {12},
\eeq

\beq
\widetilde N_{DN}(t)=\frac 1 {4\pi} \frac {\sqrt 3} 8 t+\frac
{-3+\sqrt 3} {8\pi} t^{1/2}-\frac 1 {12},
\eeq

with the analog of Theorem 2.2 holding.  Here the constant is
$\left(\psi(\pi)-\psi(\frac \pi 2)\right)+\psi(\frac\pi
3)+\left(\psi(\frac \pi 3)-\psi(\frac \pi 6)\right)=-\frac 1
{16}+\frac 2 9-\frac {35} {144}=-\frac 1 {12}.$

It does not appear that mixed boundary problems on the equilateral
triangle or the other types of mixed boundary problems on the
$30^\circ-60^\circ-90^\circ$ triangle have eigenfunctions that may all
be described by trigonometric polynomials.\\

\section{The Sphere}

We consider the 2-dimensional sphere $S^2$, and for simplicity we take
the radius equal to 1.  The eigenvalues of the Laplacian are given by
the theory of spherical harmonics, with eigenvalue $k(k+1)$ having
multiplicity $2k+1$, for $k=0,1,\ldots$  The eigenfunctions are the
restriction to $S^2$ of homogeneous harmonic polynomials of degree
$k$.  We thus have
\beq N(t)=k^2\quad\text{for}\quad k^2-k\le t<k^2+k,
\eeq
which we simplify to
\beq N(t)=\left[\sqrt{t+\frac 1 4}+\frac 1 2\right]^2.
\eeq
Now we define the refined asymptotics
\beq \widetilde N(t)=t+\frac 1 3.
\eeq
\begin{lemma} For $k^2-k\le t<k^2+k$ we have
\beq A(t)=\frac 1 {2t} \left( k^2-(k^2-t)^2\right)-\frac 1 3
\eeq
\end{lemma}
\begin{proof}
$\displaystyle{A(t)=\frac 1 t \int_0^t(N(s)-s)\,ds-\frac 1 3.}$
Now for $j<k$ we observe that\\
$\displaystyle{\int_{j^2-j}^{j^2+j}(N(s)-s)\,ds=\int_{j^2-j}^{j^2+j}(j^2-s)\,ds=0}$
because $j^2-s$ is a linear function on the interval varying between
$j$ and $-j$.  Thus $\displaystyle{A(t)=\frac 1 t
\int_{k^2-k}^t(k^2-s)\,ds-\frac 1 3}$, and when we evaluate the
integral we obtain (4.4).
\end{proof}

\begin{theorem} Let
\beq g(x)=\frac 1 6-2(x-[x+\frac 1 2])^2.
\eeq
Then $g$ is a continuous periodic function of period 1 with mean value zero, and
\beq A(t)=g(\sqrt{t+\frac 1 4})+O(t^{-1/2}).
\eeq
\end{theorem}
\begin{proof} It is clear from the definition that $g$ is continuous
and periodic, and if $x=k+r$ with $-\frac 1 2 \le r <\frac 1 2$ then
$g(x)=\frac 1 6 - 2r^2.$  One easily checks that
$\displaystyle{\int_{-\frac 1 2}^{\frac 1 2}(\frac 1 6 -
2r^2)\,dr=0}$, so $g$ has a mean value zero.
\end{proof}

If we write $\sqrt{t+\frac 1 4}=k+r$ then $k^2=t+\frac 1 4
-2r\sqrt{t+\frac 1 4}+r^2$ and $(k^2-t)^2=\left(\frac 1 4 +
r^2-2r\sqrt{t+\frac 1 4}\right)^2=(\frac 1 4 +
r^2)^2-r(1+4r^2)\sqrt{t+\frac 1 4}+4r^2(t+\frac 1 4)$\vspace{1mm} so
$k^2-(k^2-t)^2=t(1-4r^2)-\sqrt{t+\frac 1 4}(1-4r^2)+\frac 1 4
+r^2-(\frac 1 4 + r^2)^2 - r^2$ $=t(1-4r^2)-\sqrt{t+\frac 1 4} r
(1-4r^2)+\frac {(4r^2+3)(1-4r^2)}{16}.$ Thus (4.4) becomes
\beq A(t)=\frac 1 6 - 2r^2 - \frac {r(1-4r^2)}{2}\frac {\sqrt{t+\frac
1 4}} {t} + \frac {(4r^2+3)(1-4r^2)}{32t}.
\eeq
\begin{cor}
There exist continuous periodic functions $g_1$ and $g_2$ such that
\beq A(t)=g\left(\sqrt{t+\frac 1 4}\right)+g_1\left(\sqrt{t+\frac 1
4}\right)\frac{\sqrt{t+\frac 1 4}} t+\frac {g_2\left(\sqrt{t+\frac 1
4}\right)} t.
\eeq
More precisely,
\beq g_1(x)=\frac {\left(x-\left[x+\frac 1
2\right]\right)\left(4\left(x-\left[x-\frac 1
2\right]\right)^2\right)} 2 \quad{\text and}
\eeq
\beq g_2(x)=\frac 1 {32} 4\left(x-\left[x+\frac 1
2\right]\right)^2+3)\left(1-4\left(x-\left[x+\frac 1
2\right]\right)^2\right).
\eeq
\end{cor}
\begin{proof}
When $x=\sqrt{t+\frac 1 4}$ we have $r=x-\left[x+\frac 1 2\right]$, so
(4.8) is exactly (4.7).
\end{proof}

Next we consider the hemisphere $H$, with either Neumann or Dirichlet
boundary conditions.  Reflecting an eigenfunction symmetrically or
skew-symmetrically about the boundary produces an eigenfunction on
$S^2$, so the eigenvalues are the same $k(k+1)$, and the multiplicity
$2k+1$ on $S^2$ splits into $k+1$ for Neumann and $k$ for Dirichlet
boundary conditions on $H$.  The analog of (4.1) is thus
\beq N_N(t)=\frac {k^2+k} 2,\quad N_D(t)=\frac{k^2-k} 2.
\eeq
For this reason we choose
\beq \left\{\begin{array}{l}
\widetilde N_N(t)=\frac 1 2 t +\frac 1 2 \sqrt{t+\frac 1 4}+\frac 1 6\\
\widetilde N_D(t)=\frac 1 2 t -\frac 1 2 \sqrt{t+\frac 1 4}+\frac 1 6.\\
\end{array}\right.
\eeq
We note that we could replace $\sqrt{t+\frac 1 4}$ by $\sqrt t$, as
that would only change the refined asymptotics by $O(t^{-1/2})$ and
this would get absorbed into the remainder term in $A(t)$.
\begin{theorem}  Let $g(t)$ be as in Theorem 4.2.  Then
\beq A(t)=\frac 1 2 g\left(\sqrt{t+\frac 1 4}\right)+O(t^{-1/2})\eeq
for either boundary condition.
\end{theorem}
\begin{proof} We may write $A(t)=\frac 1 2 A_1(t)\pm \frac 1 2 A_2(t)$
where $A_1$ is the average function for the sphere, and
\[A_2(t)=\frac 1 t \int^t_0\left(j-\sqrt{s+\frac 1 4}\right)\,ds\]
where $\sqrt{s+\frac 1 4}=j+r$ with $-\frac 1 2 \le r <\frac 1 2$.  We
break the integral up into the intervals $j^2-j\le s <j^2+j$ for
$j<k$, and the final interval $k^2-k\le s \le t$.\vspace{1mm}  Now
$\displaystyle \int^{j^2+j}_{j^2-j} \left(j-\sqrt{s+\frac 1
4}\right)\,ds=\int^{\frac 1 2}_{-\frac 1 2} 2r(j+r)\,dr=\frac 1 6$, so
the contributions from the intervals $j<k$ add up to $\frac 1
{6t}(k-1)=O(t^{-1/2})$.  For the final interval we note that the
integrand is bounded and the interval has length at most $2k$, so its
contribution is also $O(\frac k t)=O(t^{-1/2})$.  Thus
$A_2(t)=O(t^{-1/2})$.\end{proof}

We note that the $\sqrt{t+\frac 1 4}$ coefficient in (4.12) is exactly
$\frac 2 {4\pi}$ where $L=2\pi$ is the length of the boundary circle
of $H$.  The circle does not contribute to the constant terms because
it is a geodesic in the sphere.\\

Finally we consider the projective sphere $PS^2$ obtained by
identifying antipodal points in $S^2$.  Then the eigenfunctions are
exactly the spherical harmonics of even degree.  Thus for $k^2-k\le t<
k^2-k$ we have
\beq N(t)=\sum_{j=0}^{\left[\frac {k-1} 2\right]}(4j+1)=
\left\{\begin{array}{ll}\frac {k^2+k} 2 & \text{ if $k$ is even.}\\
\frac {k^2-k} 2 & \text{ if $k$ is odd.}\end{array}\right.
\eeq
Thus we define the refined asymptotics
\beq \widetilde N(t)=\frac 1 2 t +\frac 1 6,
\eeq
and the analog of Theorem 4.4 holds, since the average value of the
function $(-1)^kk$ is $O(t^{-1/2})$.\\

\section{Lunes}

Fix an even number $2m$, and consider the lune $L_{2m}$ cut from the
sphere by two great circles from pole to pole separated by the angle
$\frac {2\pi} {2m}$.  Let $(x,y,w)$ denote the coordinates in $\mathbb
R^3$.  We may also think of spherical harmonics as polynomials in
$(z,\bar z, w)$.  Eigenfunctions on the lune with Neumann boundary
conditions extend by even reflection to the whole sphere, and so are
given by spherical harmonics that are invariant under the rotation
$z\mapsto e^{\frac{2\pi i}{m}}z$ and the reflection $z\mapsto \bar z$.
 Similarly, with Dirichlet boundary conditions and odd reflection the
spherical harmonics muct be invariant under the rotation and
skew-symmetric under the reflection.\\
\newcommand{\PN}{\mathbb P_N}
\newcommand{\QN}{\mathbb Q_N}

Fix a nonnegative integer $N$, and let $\PN$ denote the space of
polynomial (not necessarily harmonic) that are homogeneous of degree
$N$ and satisfy the conditions
\beq f(e^{\frac{2\pi i} m} z, e^{-\frac {2\pi i} m}\bar z, w)=f(z,\bar
z, w)\qquad \text{and}\eeq
\beq f(\bar z,z,w)=f(z,\bar z, w).\eeq
It is clear that the functions $|z|^{2j}(z^m+\bar z^m)^kw^\ell$, for
$2j+mk+\ell=N,$ belong to $\PN$, and in fact give a basis for $\PN$.
Similarly  if $\QN$ denotes the space of polynomials homogeneous of
degree $N$ satisfying (5.1) and
\beq f(\bar z,z,w)=-f(z,\bar z,w)\eeq
instead of (5.2) then $\QN$ consists of functions of the form
$(z^m-\bar z^m)f_1(z,\bar z, w)$ for $f_1\in \mathbb P_{N-m}$ (in
particular $\QN$ is $\{0\}$ unless $N\ge m$).  We also observe that
$\Delta:\PN\rightarrow\mathbb P_{N-2}$ and $\Delta:\QN\rightarrow
\mathbb Q_{N-2}$ and both mappings are onto.  It follows that the
multiplicities of the eigenvalue $N(N+1)$ in the Neumann and Dirichlet
spectra of $L_{2m}$ are given by $\dim \PN-\dim \mathbb P_{N-2}$ and
$\dim \QN-\dim \mathbb Q_{N-2}$ respectively.  It turns out to be
easier to compute these differences of dimensions rather than $\dim \PN$ and
$\dim \QN$.

\begin{lemma} For any $N\ge 0$ we have
\beq \dim \PN-\dim \mathbb P_{N-2}=\left[\frac N m \right]+1,\eeq
while for any $N\ge m$ we have
\beq \dim \QN-\dim \mathbb Q_{N-2}=\left[\frac N m \right].\eeq\end{lemma}

\begin{proof} The mapping $f\mapsto |z|^2f$ maps $\mathbb P_{N-2}$
into $\PN$ and is one-to-one, so $\dim \PN-\dim \mathbb P_{N-2}$ is
equal to the dimension of the complementary space, which is spanned by
$(z^m+\bar z^m)^kw^\ell$ for $mk+\ell=N$.  For each choice of $k$ with
$mk\le N$ there is a unique choice of $\ell$, namely $\ell=N-mk.$  But
the number of choices of $k$ with $mk\le N$ is exactly $\left[\frac N
m\right]+1$, proving (5.4).  Then\vspace{1mm} $\dim\QN-\dim \mathbb
Q_{N-2}=\dim\mathbb P_{N-m}-\dim\mathbb P_{N-m-1}=\left[\frac {N-m}
m\right]+1=\left[\frac N m\right]$.
\end{proof}
Let $N_N(t)$ and $N_D(t)$ denote the eigenvalue counting functions for
the Neumann and Dirichlet eigenfunctions on $L_{2m}$.  Then for
$k^2-k\le t < k^2+k$ we have by Lemma 5.1 that
\beq N_N(t)=\sum_{j=0}^{k-1}\left(\left[\frac j m
\right]+1\right)\qquad \text{and}
\eeq
\beq N_D(t)=\sum_{j=m}^{k-1}\left[\frac j m \right]
\eeq
\begin{lemma} Suppose $k\equiv p \mod m$ and $k^2-k\le t < k^2 + k$.  Then
\beq N_N(t)=\frac {k^2} {2m}+\frac k 2+\frac {p(m-p)}{2m}\qquad \text{and}
\eeq
\beq N_D(t)=\frac {k^2} {2m}-\frac k 2+\frac {p(m-p)}{2m}.
\eeq
\end{lemma}
\begin{proof} The sum in (5.6) repeats each integer from 1 up to
$\left[\frac k m \right]$ exactly $m$ times, and then $\left[\frac k m
\right]+1$ exactly $p$ times.  Note that $k=p+m\left[\frac k m
\right]$, so \vspace{1mm} $N_N(t)=\frac m 2 \left[\frac k m
\right]\left(\left[\frac k m \right]+1\right)+p\left(\left[\frac k m
\right]+1\right)$$=\left(\left[\frac k m \right]+1\right)\left(\frac m
2 \left[\frac k m \right]+p\right)$\vspace{1mm}$=\left(\frac {k-p+m} m
\right)\left(\frac {k+p} 2\right)$ which  yields (5.8).  Then (5.9)
follows since $N_D(t)=N_N(t)-k$ by (5.6) and (5.7).
\end{proof}

Now we note that the average value of the constant term $\frac
{p(m-p)}{2m}$ as $p$ varies from $0$ to $m-1$ is
\begin{eqnarray*} \frac 1 {2m^2}\sum^{m-1}_{p=0} (mp-p^2) &=& \frac
{m^2(m-1)} {4m^2} - \frac {m(m-1)(2m-1)}{12m^2} \\ &=& \frac 1 {12}
\left(m-\frac 1 m\right).
\end{eqnarray*}
This leads to the choice of refined asymptotics
\beq
\widetilde N(t)=\frac t {2m}\pm\frac {\sqrt{t+\frac 1 4}} 2 + \frac 1
{12} \left(m-\frac 1 m\right)+\frac 1 {6m}
\eeq
($+$ for $\widetilde N_N$ and $-$ for $\widetilde N_D$).  Note that we
could also replace $\sqrt{t+\frac 1 4}$ by $\sqrt t$ since the
difference is $O(t^{-1/2})$.
\begin{theorem} There exists a continuous, periodic function $g$ of
mean value zero such that
\beq A(t)=g(\sqrt{t+\frac 1 4})+O(t^{-1/2}).
\eeq
\end{theorem}
\begin{proof} We may write $D(t)=N(t)-\widetilde N(t)$ as
$D_1(t)+D_2(t)+D_3(t)$ where
\beq \left\{\begin{array}{lcl}
D_1(t) &=& \frac {k^2} {2m}-\left(\frac t {2m}+\frac 1
{6m}\right),\vspace{1mm}\\
D_2(t) &=& \pm\left(\frac {k} {2}-\frac {\sqrt {t+\frac 1 4}} 2
\right),\qquad\text{and}\\
D_3(t) &=& \frac {p(m-p)}{2m}-\frac 1 {12}\left(m-\frac 1
m\right).\end{array}\right.\eeq
If we form the corresponding average functions $A_1,A_2,A_3$ it
suffices to show (5.11) for each separately.  Now the result for $A_1$
is Theorem 4.2 multiplied by $\frac 1 {2m}$.  For $A_2$ it is easy to
see that in fact $A_2(t)=O(t^{-1/2})$.  It remains to consider
$A_3$.\\

Assume $k^2-k\le t< k^2+k$ for $k=jm+p$ with $1\le p\le m$.  We may write
\[ A_3(t)=\frac 1
t\sum_{n=0}^{j-1}\sum_{q=1}^m\int_{(nm+q)^2-(nm+q)}^{(nm+q)^2+(nm+q)}
D_3(t)\,dt\,+\frac 1 t\int ^t_{(jm+1)^2-(jm+1)} D_3(t)\,dt.\]
It is easy to see that the last term is $O(t^{-1/2})$ because the
integrand is bounded and the length of the interval is $O(t^{1/2})$.
In each integral in the sum $D_3(t)$ is constant, so
\begin{eqnarray*} A_3(t) &=&\frac 1 t \sum_{n=0}^{j-1}\sum^m_{q=1}
2(nm+q)\left(\frac {q(m-q)} {2m} - \frac 1 {12} \left(m-\frac 1
m\right)\right)+O(t^{-1/2})\\
&=&\frac 1 t \sum_{n=0}^{j-1}\sum^m_{q=1} 2q\left(\frac{q(m-q)}
{2m}-\frac 1 {12}\left(m-\frac 1 m\right)\right)+O(t^{-1/2})\\
\end{eqnarray*}
since $\displaystyle \sum_{q=1}^m 2nm\left(\frac {q(m-q)} {2m} -\frac
1 {12} \left(m-\frac 1 m\right)\right)=0$. But the summands are
uniformly bounded and the number of terms is $O(t^{1/2})$, so
$A_3(t)=O(t^{-1/2}).$
\end{proof}

We observe (5.10) has the predicted form.  Conpared with (4.3) for the\vspace{1mm}
sphere, the leading term is reduced by the factor $\frac 1 {2m}$\vspace{1mm}
because the area of the lune is $\frac 1 {2m}$ times the area \vspace{1mm}of the
circle, and the $\pm\frac 1 2$ factor in the square root term is
$\frac {2\pi} {4\pi}$, where $2\pi$ is the length of the perimeter of
the lune.  Finally, the constant is twice $\frac 1 {24}\left(m-\frac 1
m\right)$, the value predicted for each of the two angels $\frac \pi
m$ of the lune.  We could also obtain the analog of corollary 4.3 for
the remainder term in (5.11).\\

Next we consider the half-lune $L_{2m}^{1/2}$ obtained by slicing
along the equator.  Functions satisfying Neumann or Dirichlet
conditions on the equatorial edge extend to the entire lune by even or
odd reflection; so we may study four different types of boundary
conditions.  We denote by $N^+_N$, $N_N^-$, $N_D^+$ and $N^-_D$ the
corresponding eigenvalue counting functions, where the $\pm$ denote
the equatorial boundary conditions.  Then $N_N(t)=N_N^+(t)+N_N^-(t)$
and $N_D(t)=N^+_D(t)+N_D^-(t)$, and we need to understand how the
spherical harmonics that contribute to $N_N(t)$ and $N_D(t)$ split
into even and odd functions in the $w$ variable.  It
is clear that this is determined by the parity of $\ell$ in the basis element
$|z|^{2j}(z^m+\bar z^m)^kw^\ell$.\\

\begin{lemma} (a) Assume $m$ is even.  Then

\beq \dim \mathbb P^+_N-\dim \mathbb P^+_{N-2}=\left\{\begin{array}{cl} \left[\frac N m\right]+1 & \text{if $N$ is even}\\
0 & \text{if $N$ is odd}\end{array}\right.
\eeq 
\beq\dim \mathbb P^-_N-\dim \mathbb P^-_{N-2}=\left\{\begin{array}{cl} 0  & \text{if $N$ is even}\\
\left[\frac N m\right]+1 & \text{if $N$ is odd}\end{array}\right.
\eeq 
\beq\dim \mathbb Q^+_N-\dim \mathbb Q^+_{N-2}=\left\{\begin{array}{cl} \left[\frac N m\right] & \text{if $N$ is even}\\
0 & \text{if $N$ is odd}\end{array}\right.
\eeq \beq\dim \mathbb Q^-_N-\dim \mathbb Q^-_{N-2}=\left\{\begin{array}{cl} 0 & \text{if $N$ is even}\\
\left[\frac N m\right] & \text{if $N$ is odd}\end{array}\right.
\eeq

(b) Assume $m$ is odd.  Then
\beq \dim \mathbb P^+_N-\dim \mathbb P^+_{N-2}=\left\{\begin{array}{cc} 
\frac 1 2\left[\frac N m\right]+\frac 1 2 & \text{if $\left[\frac N m\right]$ is odd}\vspace{1mm}\\
\frac 1 2\left[\frac N m\right]+1 & \text{if $\left[\frac N m\right]$ is even and $N$ is even}\vspace{1mm}\\
\frac 1 2\left[\frac N m\right] & \text{if $\left[\frac N m\right]$ is even and $N$ is odd}\end{array}\right.
\eeq 
\beq \dim \mathbb P^-_N-\dim \mathbb P^-_{N-2}=\left\{\begin{array}{cc} 
\frac 1 2\left[\frac N m\right]+\frac 1 2 & \text{if $\left[\frac N m\right]$ is odd}\vspace{1mm}\\
\frac 1 2\left[\frac N m\right] & \text{if $\left[\frac N m\right]$ is even and $N$ is even}\vspace{1mm}\\
\frac 1 2\left[\frac N m\right] +1 & \text{if $\left[\frac N m\right]$ is even and $N$ is odd}\end{array}\right.
\eeq 
\beq \dim \mathbb Q^+_N-\dim \mathbb Q^+_{N-2}=\left\{\begin{array}{cc} 
\frac 1 2\left[\frac N m\right]& \text{if $\left[\frac N m\right]$ is even}\vspace{1mm}\\
\frac 1 2\left[\frac N m\right]+\frac 1 2 & \text{if $\left[\frac N m\right]$ is odd and $N$ is odd}\vspace{1mm}\\
\frac 1 2\left[\frac N m\right]  - \frac 1 2& \text{if $\left[\frac N m\right]$ is odd and $N$ is even}\end{array}\right.
\eeq 
\beq \dim \mathbb Q^-_N-\dim \mathbb Q^-_{N-2}=\left\{\begin{array}{cc} 
\frac 1 2\left[\frac N m\right]& \text{if $\left[\frac N m\right]$ is even}\vspace{1mm}\\
\frac 1 2\left[\frac N m\right]-\frac 1 2 & \text{if $\left[\frac N m\right]$ is odd and $N$ is odd}\vspace{1mm}\\
\frac 1 2\left[\frac N m\right] + \frac 1 2& \text{if $\left[\frac N m\right]$ is odd and $N$ is even}\end{array}\right.
\eeq 
\end{lemma}
\begin{proof}
As in the proof of Lemma 5.1, we need to count the number of solutions of $mk+\ell=N$ but now with $\ell$ restricted to be even and odd.  If $m$ is even then $\ell=N-mk$ will have the same parity as $N$, and all $\left[\frac N m\right]+1$ choices of $k$ are allowed.  This proves (5.13) and (5.14).  A similar argument proves (5.15) and (5.16) since there are only $\left[\frac N m\right]$ choices of $k$.  If $m$ is odd, then $\ell$ will be even if $N$ and $k$ have the same parity, and $\ell$ will be odd if $N$ and $k$ have opposite parity.  If $N$ is even, then the even choices of $k$ in $0\le k \le \left[\frac N m\right]$ will contribute to $\dim \PN^+-\dim\mathbb P^+_{N-2}$ while the odd choices of $k$ will contribute to $\dim \mathbb P^-_N-\dim \mathbb P^-_{N-2}$.  If $\left[\frac N m\right]$ is odd there will be exactly $\frac 1 2\left[\frac N m\right]+\frac 1 2$ of even and odd choices, while if $\left[\frac N m\right]$ is even there will be $\frac 1 2 \left[\frac N m\right]+1$ even choices and $\frac 1 2 \left[\frac N m\right]$ odd choices.  If, on the other hand, $N$ is odd, it is the odd choices of $k$ that contribute to $\dim \mathbb P^+_N-\dim \mathbb P^+_{N-2}$ and the even choices of $k$ that contribute to $\dim \mathbb P^-_N-\dim \mathbb P^-_{N-2}$.  Putting this together yields (5.17) and (5.18).  Then (5.19) and (5.20) are obtained by replacing $N$ by $N-m$, which flips the parity of both $N$ and $\left[\frac N m\right]$.
\end{proof} 
\begin{lemma} Suppose $k\equiv p\mod m$ and $k^2-k\le t< k^2+k$.\vspace{1mm}\\
(a) Assume $m$ is even.  Then
\beq N_N^+(t)=\left\{\begin{array}{cc} 
\frac {k^2} {4m} + \frac k 4 +\frac {(m-p)p}{4m} & \text{if $p$ is even}\vspace{1mm}\\
\frac {k^2} {4m} +\left(\frac 1 4 +\frac 1 {2m}\right)k+\frac {(m-p)p}{4m} +\frac {m-p} {2m} & \text{if $p$ is odd}
\end{array}\right.\eeq
\beq N_N^-(t)=\left\{\begin{array}{cc} 
\frac {k^2} {4m} + \frac k 4 +\frac {(m-p)p}{4m} & \text{if $p$ is even}\vspace{1mm}\\
\frac {k^2} {4m} +\left(\frac 1 4 -\frac 1 {2m}\right)k+\frac {(m-p)p}{4m} +\frac {m-p} {2m} & \text{if $p$ is odd}
\end{array}\right.\eeq
\beq N_D^+(t)=\left\{\begin{array}{cc} 
\frac {k^2} {4m} - \frac k 4 +\frac {(m-p)p}{4m} & \text{if $p$ is even}\vspace{1mm}\\
\frac {k^2} {4m} -\left(\frac 1 4 -\frac 1 {2m}\right)k+\frac {(m-p)p}{4m} -\frac {p} {2m} & \text{if $p$ is odd}
\end{array}\right.\eeq
\beq N_D^-(t)=\left\{\begin{array}{cc} 
\frac {k^2} {4m} - \frac k 4 +\frac {(m-p)p}{4m} & \text{if $p$ is even}\vspace{1mm}\\
\frac {k^2} {4m} -\left(\frac 1 4 +\frac 1 {2m}\right)k+\frac {(m-p)p}{4m} +\frac {p} {2m} & \text{if $p$ is odd}
\end{array}\right.\eeq

(b) Assume $m$ is odd.  Then
\beq N^+_N(t)=\frac {k^2}{4m}+\left(\frac 1 4 + \frac 1 {4m}\right)k+\frac {p(m-p)}{4m}+\frac {m-p}{4m}+h(n,p)
\eeq
\beq N^-_N(t)=\frac {k^2}{4m}+\left(\frac 1 4 - \frac 1 {4m}\right)k+\frac {p(m-p)}{4m}-\frac {m-p}{4m}-h(n,p)
\eeq
\beq N^+_D(t)=\frac {k^2}{4m}+\left(-\frac 1 4 + \frac 1 {4m}\right)k+\frac {p(m-p)}{4m}+\frac {m-p}{4m}+\frac 1 2+h(n,p)
\eeq
\beq N^-_D(t)=\frac {k^2}{4m}+\left(-\frac 1 4 - \frac 1 {4m}\right)k+\frac {p(m-p)}{4m}-\frac {m-p}{4m}+\frac 1 2 -h(n,p)
\eeq
where
\beq h(n,p)=\left\{\begin{array}{cl}
0 & \text{if $n$ is odd}\\
\frac 1 4 & \text{if $n$ is even and $p$ is odd}\\
-\frac 1 4 & \text{if $n$ is even and $p$ is even.} 
\end{array}\right.\eeq
\end{lemma}
\begin{proof} As in the proof of Lemma 5.2, we need to sum the expressions in Lemma 5.4 for $N\le k-1$.  Assume $m$ is even.  Write $k=mn+p$.  Then
\[N_N^+(t)=\sum_{j\le\frac {k-1} 2} \left(\left[\frac {2j} m\right]+1\right)\quad\text{and}\]
\[N_N^-(t)=\sum_{j\le\frac {k-2} 2} \left(\left[\frac{2j+1} m\right] + 1\right).\]
The sum for $N_N^+$ has all integers from 1 to $n$ repeated $\frac m 2$ times, and the integer $n+1$ repeated $\frac p 2$ times if $p$ is even and $\frac {p+1} 2$ times if $p$ is odd. Thus
\[N_N^+(t)=\left\{\begin{array}{cc}
\frac m 4 n(n+1) + \frac p 2 (n+1) & \text{if $p$ is even}\\
\frac m 4 n(n+1) + \frac {p+1} 2 (n+1) & \text{if $p$ is odd.}\\
\end{array}\right.\]
Substituting $n=\frac{k-p} m$ we obtain (5.21). Similar reasoning shows
\[N_N^-(t)=\left\{\begin{array}{cc}
\frac m 4 n(n+1) + \frac p 2 (n+1) & \text{if $p$ is even}\\
\frac m 4 n(n+1) + \frac {p-1} 2 (n+1) & \text{if $p$ is odd.}\\
\end{array}\right.\]
and this reduces to (5.22).\\

Similarly we have
\[N_D^+(t)=\sum_{j\le\frac {k-1} 2} \left[\frac {2j} m\right]\quad\text{and}\]
\[N_D^-(t)=\sum_{j\le\frac {k-2} 2} \left[\frac{2j+1} m\right]\]
so we just have to replace $n$ by $n-1$ in the expressions for $N_N^+(t)$ and $N_N^-(t)$.  This yields (5.23) and (5.24).\\

Next consider the case when $m$ is odd.  For $N_N^+(t)$ we need to sum over $N$ in $0\le N\le k-1$ the values $\frac 1 2 \left[\frac N m\right]+\frac 1 2$, $\frac 1 2 \left[\frac N m\right]+1$ or $\frac 1 2 \left[\frac N m\right]$ according to the parity cases in (5.17).  As before there will be $m$ choices of $N$ with $\left[\frac N m\right]=j$ for $0\le j\le n-1$, and $p$ choices of $N$ with $\left[\frac N m\right]=n$.  So \[\begin{split}N_N^+(t)=m\sum_{j=0}^{n-1}\frac 1 2 j &+\frac 1 2 pn + \frac 1 2 \#\{N:\left[\frac N m\right] \text{ is odd}\}\\
&+\#\{N:\left[\frac N m\right] \text{ is even and $N$ is even}\}\end{split}\]
If $n$ is odd then 
\[\#\{N:\left[\frac N m\right]\text{ is odd}\}=m\left(\frac{n-1} 2\right)+p\]
and \[\quad\#\{N:\left[\frac N m\right] \text{ is even and $N$ is even}\}=\left(\frac{m+1} 2\right)\left(\frac{n+1} 2\right).\]  In this case
\[\begin{split}N_N^+(t)&=\frac{m(n-1)n} 4 + \frac {pn} 2 +\frac {m(n-1)+(m+1)(n+1)} 4 +\frac p 2\\&=\frac {m(n+1)n} 4 +\frac {n+1} 4+\left(\frac {n+1} 2\right)p,\end{split}\]
and substituting $n=\frac{k-p} m$ we obtain
\[N_N^+(t)=\frac {k^2}{4m}+\left(\frac{m+1}{4m}\right)k+\frac{p(m-p)}{4m}+\frac{m-p}{4m}.\]
If $n$ is even then
\[\#\{N:\left[\frac N m\right]\text{ is odd}\}=\frac {mn} 2, \quad\text{and}\]
\[\#\{N:\left[\frac N m\right]\text{ is even and $N$ is even}\}=\left\{\begin{array}{cc}
\left(\frac{m+1} 2\right)\frac n 2 + \frac{p+1} 2 & \text{ if $p$ is odd}\\
\left(\frac{m+1} 2\right)\frac n 2 + \frac{p} 2 & \text{ if $p$ is even.}\end{array}\right.\]
So 
\[\begin{split}N_N^+(t)&=\frac {m(n-1)n} 4 +\frac {pn} 2 +\frac {mn+(m+1)n} 4 + \left\{\begin{array}{cc} \frac {p+1} 2 & \text{ $p$ odd}\\ \frac p 2 & \text{ $p$ even}\end{array}\right.\\
&=\frac {m(n+1)n} 4 +\frac n 4 +\left(\frac{n+1} 2\right)p+\frac 1 2\,(\text{if $p$ is odd}).\end{split}\]
This establishes (5.25), and (5.26) follows from $N_N(t)=N_N^+(t)+N_N^-(t)$ and (5.8).\\

To compute $N_D^\pm(t)$ we note that (5.19) and (5.20) differ from (5.17) and (5.18) by $-\frac 1 2$.  Thus $N_D^+(t)=N_N^+(t)-\left(\frac{k-1}2\right)$ and $N_D^-(t)=N_N^-(t)-\left(\frac{k-1}2\right)$, and this implies (5.27) and (5.28).
\end{proof}

Now we define the refined asymptotics
\beq \widetilde N_N^+(t)=\frac t {4m} +\left(\frac 1 4 +\frac 1 {4m}\right)\sqrt{t+\frac 1 4}+\frac 1 {24}\left(m-\frac 1 m\right)+\frac 1 8 + \frac 1 {12m}
\eeq
\beq \widetilde N_N^-(t)=\frac t {4m} +\left(\frac 1 4 -\frac 1 {4m}\right)\sqrt{t+\frac 1 4}+\frac 1 {24}\left(m-\frac 1 m\right)-\frac 1 8 + \frac 1 {12m}
\eeq
\beq \widetilde N_D^+(t)=\frac t {4m} +\left(-\frac 1 4 +\frac 1 {4m}\right)\sqrt{t+\frac 1 4}+\frac 1 {24}\left(m-\frac 1 m\right)-\frac 1 8 + \frac 1 {12m}
\eeq
\beq \widetilde N_D^-(t)=\frac t {4m} +\left(-\frac 1 4 -\frac 1 {4m}\right)\sqrt{t+\frac 1 4}+\frac 1 {24}\left(m-\frac 1 m\right)+\frac 1 8 + \frac 1 {12m}.
\eeq

Note that the coefficients of $\sqrt{t+\frac 1 4}$ may be explained because the two quarter circle boundary arcs have total length $\pi$, contributing $\frac 1 4$ to $\widetilde N_N^\pm$ and $-\frac 1 4$ to $\widetilde N_D^\pm$, and the equatorial boundary arc has length $\frac \pi m$, contributing $\frac 1 {4m}$ to $\widetilde N_N^+$ and $\widetilde N_D^+$ and $-\frac 1 {4m}$ to $\widetilde N_N^-$ and $\widetilde N_D^-$.  For the constant term we note that the top angle $\frac \pi m$ always has the same type of boundary conditions on either side, so it contributes $\frac 1 {24}\left(m-\frac 1 m\right)$, while the bottom right angles have like boundary conditions for $\widetilde N_N^+$ and $\widetilde N_D^-$, so contribute $\frac 1 8$ to each, and opposite boundary conditions for $\widetilde N_N^-$ and $\widetilde N_D^+$ and so contribute $-\frac 1 8$ to each.\\
\begin{theorem} The analog of Theorem 5.3 holds. \end{theorem}
\begin{proof} If suffices to understand the contribution of the parity dependent terms in Lemma 5.5.  Suppose first that $m$ is even. When we average the function
\[\phi(p)=\left\{\begin{array}{cc} 0 & \text{ if $p$ is even}\\
\frac {m-p} {2m} & \text{ if $p$ is odd}\end{array}\right.\]
over $1\le p \le m$ we obtain
\[\frac 1 m \sum_{j=0}^{\frac {m-1} 2}\frac {2j+1}{2m}=\frac 1 {2m^2}\left(\frac m 2\right)^2=\frac 1 8\]
so the average of $\phi(p)-\frac 1 8$ is $O(t^{-1})$.\\

Next suppose $m$ is odd.  This time we have to average the functions
\[\phi_0(p)=\frac {m-p}{4m}\]
when $n$ is odd, and
\[\phi_1(p)=\frac {m-p}{4m}+\left\{\begin{array}{cc} \frac 1 4 & \text{ if $p$ is odd}\\
-\frac 1 4 & \text{ if $p$ is even}\end{array}\right.\]
when $n$ is even.  Now the average value of $\phi_0$ is $\frac {m(m-1)}{8m^2}=\frac 1 8 - \frac 1 {8m}$, and the average value of $\phi_1$ is $\frac 1 8 - \frac 1 {8m} + \frac 1 {4m}=\frac 1 8 + \frac 1 {8m}$ because there is exactly one more odd value of $p$ than even values of $p$ in $1\le p \le m$.  Thus when we average over $n$ we get $\frac 1 8$.
\end{proof}
\section{Cone point singularities}

The first example of a surface with cone point singularities is the flat projective plane obtained from the unit square by identifying both pairs of opposite edges with reversed orientation.  This yields two singular points with cone angle $\pi$.  This example was analyzed in [JS] where it was shown that
\beq N(t)=\frac 1 4 N_T(t)+\frac 1 4 \pm \frac 1 2
\eeq
where $T$ is the $2\times 2$ torus, and the choice of $\pm$ sign corresponds to the parity of $\left[\frac {t^{1/2}} \pi\right]$.  Thus the choice
\beq \widetilde N(t)=\frac 1 {4\pi}t+\frac 1 4
\eeq
leads to the analog of Theorem 2.2.  We note that the two cone points yield $2\cdot 2 \psi(\frac \pi 2)=\frac 1 4$ for the constant in (6.2).\\

Our next example is the surface of a regular tetrahedron.  This is discussed in detail in [GKS].  Let $T$ be the torus associated to the lattice $\mathscr L$ generated by the vectors $(2,0)$ and $(1,\sqrt 3)$.  This is similar to the hexagonal torus discussed in section 2 but it is larger, and a fundamental domain consists of eight equilateral triangles.  There is a two-fold covering of the tetrahedron by $T$ (deleting singular points), and so the eigenfunctions on $T$ have the form
\beq e(k,j)=e^{2\pi i(ku+jv)\cdot x}
\eeq
for $u=(\frac 1 2, \frac {\sqrt 3} 6)$ and $v=(0,\frac {\sqrt 3} 3)$, generators of the dual lattice $\mathscr L^*$, and $(k,j)\in\mathbb Z^2$, with eigenvalue
\beq \frac {4\pi^2} 3 (j^2+k^2+jk),
\eeq
while the eigenfunctions on the tetrahedron are
\beq \cos 2\pi (ku+jv)\cdot x = \cos 2\pi \left(\frac k 2 x_1+\left(\frac {\sqrt 3} 3 j +\frac {\sqrt 3} 6 k \right)x_2\right).
\eeq
Aside from $(k,j)=(0,0)$, the two choices $(k,j)$ and $(-k,-j)$ collapse to one choice in (6.5), so
\beq N(t)=\frac 1 2 N_T(t)+\frac 1 2
\eeq
This leads to the choice
\beq \widetilde N(t)=\frac {\sqrt 3} {4\pi}t+\frac 1 2
\eeq
and the analog of Theorem 2.2 holds.  Note that the tetrahedron has four cone points with cone angles $\pi$, and $4\cdot 2\psi(\frac\pi 2)=\frac 1 2$.\\

Next we consider a half of the tetrahedron sliced so that two adjacent faces are cut by the perpendicular bisectors of their common edge.  The boundary of the half-tetrahedron consists of the cut line together with a side edge of the face that is entirely in the surface (see Figure 6.1).\\

\begin{figure}[h!]
\centering
\includegraphics[scale=.7]{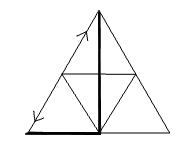}
\caption{Half-tetrahedron with boundary line in dark.}
\end{figure}

\noindent If we impose Neumann boundary conditions then for generic choice of $(k,j)$ we have eigenfunctions
\beq \cos 2\pi (ku+jv)\cdot x + \cos 2\pi (-ku+(k+j)v)\cdot x,
\eeq
while with Dirichlet boundary conditions we have
\beq \cos 2\pi (ku+jv)\cdot x - \cos 2\pi (-ku+(k+j)v)\cdot x
\eeq
corresponding to $(k,j)$ and $(-k,k+j)$ in (6.5).  However, in the nongeneric cases, if $k=0$ then we have one function of the form (6.8) and none of the form (6.9) corresponding to a single function of the form (6.5), and similarly for $(2k,-k)$.  Thus
\beq N(t)=\frac 1 4 N_T(t)+\frac 1 2 \left(\left[\frac {\sqrt 3 t^{1/2}}{2\pi}\right]+\left[\frac {t^{1/2}}{2\pi}\right]\right)+\frac 3 4
\eeq
for Neumann boundary conditions, and
\beq N(t)=\frac 1 4 N_T(t)-\frac 1 2 \left(\left[\frac {\sqrt 3 t^{1/2}}{2\pi}\right]+\left[\frac {t^{1/2}}{2\pi}\right]\right)-\frac 1 4
\eeq
for Dirichlet boundary conditions.  Thus we choose
\beq \widetilde N(t)=\frac {\sqrt 3} {8\pi} t \pm \left(\frac {(\sqrt 3 + 1)}{4\pi} t^{1/2}\right)+\frac 1 4
\eeq
(+ for Neumann and $-$ for Dirichlet) in order to obtain the analog of Theorem 2.2.  Note that $\sqrt 3 + 1$ is the length of the boundary.  To explain the constant term we have to remember that angles at corners have to be measured in terms of the intrinsic geometry of the surface.  Thus the boundary actually only has two corners, each of angle $\frac \pi 2$, so these contribute $2\cdot \frac 1 {16}=\frac 1 8$.\\

Our last set of examples are the glued lunes, where we identify the boundary arcs of a lune in the same orientation.  Here we will be able to allow any positive integer $m$ for the lune $L_m$, so the glued lune $\widetilde L_m$ will have two cone points with cone angle $\frac {2\pi} m$.  Eigenfunctions on $\widetilde L_m$ are just spherical harmonics on the sphere satisfying rotation invariance under $z\mapsto e^{\frac {2\pi i} m}z$.  As in section 5, we may identify a basis of the space $\mathbb P_N$ of homogeneous polynomials of degree $N$ with this invariance as functions of the form $|z|^{2j}z^{mk}w^\ell$ or $|z|^{2j}\bar z^{mk}w^\ell$ with $2j+mk+\ell=N$.  For any fixed $j$ and $k$ with $2j+mk\le N$ there is a unique choice of $\ell$, and there are two basis elements when $k>0$ and one when $k=0$.  By considering the map $f\mapsto |z|^2f$ from $\widetilde{\mathbb P}_{N-2}$ to $\widetilde{\mathbb P}_N$ we may compute the difference $\dim\widetilde{\mathbb P}_N-\dim\widetilde{\mathbb P}_{N-2}$ by counting the basis elements in $\widetilde {\mathbb P}_N$ corresponding to the choice $j=0$, so analogous to Lemma 5.1 we have
\beq \dim\widetilde{\mathbb P}_N-\dim\widetilde{\mathbb P}_{N-2}=2\left[\frac N m\right]+1
\eeq
Note that this is exactly the sum of (5.4) and (5.5), since the eigenspace for $\widetilde L_m$ is exactly the sum of the Neumann and Dirichlet eigenspaces of $L_{2m}$.  In particular, to get $N(t)$ for $\widetilde L_m$ we simply need to add (5.8) and (5.9), so if $k\equiv p \mod m$ and $k^2-k\le t < k^2+k$ we have
\beq N(t)=\frac {k^2} m + \frac {p(m-p)} m .
\eeq
Just as in (5.10) we choose
\beq \widetilde N(t)=\frac t m + \frac 1 6(m-\frac 1 m)+\frac 1 {3m}
\eeq
and obtain the analog of Theorem 5.3.  The contribution of the two cone angles to the constant is $2\cdot 2\psi(\frac \pi m)=\frac 1 6 (m-\frac 1 m)$.

\section{Sorting by symmetry types}

Suppose the surface has a finite group $G$ of isometries.  The group action then commutes with the Laplacian, so the eigenspaces are preserved, and we may sort the eigenfunctions by symmetry type, specifically the irreducible representations $\{\pi_j\}$ of $G$.  This situation was discussed for a more general setting in [S], where it was suggested that asymptotically, the proportion of eigenfunctions of each symmetry type should be $\displaystyle \frac {d_j^2} {\#G}$, where $d_j=\dim\pi_j$ (it is well known that $\sum d^2_j=\#G$).  This was proposed as a heuristic idea that is not universally valid (there are simple counterexamples), but should hold if there is a fundamental domain for the group whose boundary is relatively small.  In the case of surfaces there is usually a fundamental domain that is a polygon, so its boundary is one-dimensional.  Here we will be able to prove a more refined statement for some of the surfaces we have examined: the square torus, the square, the hexagonal torus and the equilateral triangle.  In the first two cases the group is the dihedral group $D_4$, and in the second two cases it is the dihedral group $D_3$.\\

The group $D_4$ has 8 elements that are generated by reflections in the diagonals of the square and the reflections in the horizontal or vertical bisectors.  There are 4 1-dimensional representations $1\pm\pm$, where the first $\pm$ indicates the symmetry type
\beq u(Rx)=\pm u(x)\eeq for diagonal reflections. and the second $\pm$ indicates the same symmetry equation where $R$ is a horizontal or vertical reflection.  There is also one 2-dimensional representation that we denote by 2.  Write $N_{\pm\pm}(t)N$ and $N_2(t)$ for the counting function of all eigenvalues $\lambda\le t$ that belong to each symmetry type.  (Note that an individual eigenspace may contain eigenfunctions of more than one symmetry type, and for $N_2(t)$ we count each eigenfunction basis element separately.)\\

For simplicity we choose a square of side length one, and this will represent the square torus if we identify opposite sides, or the square itself with either Neumann or Dirichlet boundary conditions throughout (we nix mixed boundary conditions because they do not respect the $D_4$ action).  So in all three cases we have
\beq N_{++}(t)+N_{+-}(t)+N_{-+}(t)+N_{--}(t)+N_2(t)=N(t),
\eeq
where $N(t)$ is given by $N_T(t)$ for the torus, and $N_N(t)$ and $N_D(t)$ in Lemma 2.1 with $a=b=1$.  

A fundamental domain for the $D_4$ action is a right isosceles triangle of hypotenuse length $\frac {\sqrt 2} 2$ and the equal side lengths $\frac 1 2$ that makes up an eighth of the square.  The key observation is that an eigenfunction of symmetry type $1\pm\pm$, when restricted to the triangle, satisfies Dirichlet or Neumann boundary conditions as illustrated in Figure 7.1, and conversely every triangle eigenfunction extends by the appropriate reflections to an eigenfunction on the torus or square.  Thus $N_\pm(t)$ is exactly the counting function for the triangle given by Lemma 2.3 and (3.22-3.30) with $a=\frac 1 2$.  Then we can use (7.2) to compute $N_2(t)$.
\begin{lemma} (a) For the square torus we have
\beq N_{++}(t)=\frac 1 8 N_T(t)+\frac 1 2\left(\left[\frac {t^{1/2}}{2\pi}\right]+\frac 1 2\right)+\frac 1 2 \left(\left[\frac {\sqrt 2 t^{1/2}}{4\pi}\right]+\frac 1 2\right)+\frac 3 8
\eeq
\beq N_{+-}(t)=\frac 1 8 N_T(t)-\frac 1 2\left(\left[\frac {t^{1/2}}{2\pi}\right]+\frac 1 2\right)+\frac 1 2 \left(\left[\frac {\sqrt 2 t^{1/2}}{4\pi}\right]+\frac 1 2\right)-\frac 1 8
\eeq
\beq N_{-+}(t)=\frac 1 8 N_T(t)+\frac 1 2\left(\left[\frac {t^{1/2}}{2\pi}\right]+\frac 1 2\right)-\frac 1 2 \left(\left[\frac {\sqrt 2 t^{1/2}}{4\pi}\right]+\frac 1 2\right)-\frac 1 8
\eeq
\beq N_{--}(t)=\frac 1 8 N_T(t)-\frac 1 2\left(\left[\frac {t^{1/2}}{2\pi}\right]+\frac 1 2\right)-\frac 1 2 \left(\left[\frac {\sqrt 2 t^{1/2}}{4\pi}\right]+\frac 1 2\right)+\frac 3 8
\eeq
\beq N_2(t)=\frac 1 2 N_T(t)-\frac 1 2
\eeq

\begin{figure}[h!]
\centering
\includegraphics[scale=.6]{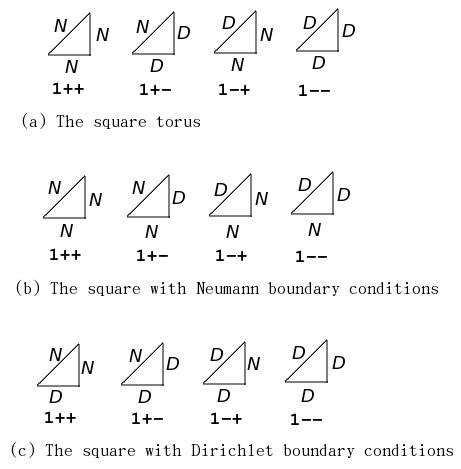}\caption{}
\end{figure}
\noindent(b) For the square with Neumann boundary conditions $N_{++}(t)$ is given by (7.3), $N_{-+}(t)$ is given by (7.5),
\beq N_{+-}(t)=\frac 1 2 N_{MM}(t)+\frac 1 2 \left(\left[\frac {\sqrt 2} {4\pi} t^{1/2}\right]+\frac 1 2 \right)
\eeq
\beq N_{--}(t)=\frac 1 2 N_{MM}(t)-\frac 1 2 \left(\left[\frac {\sqrt 2} {4\pi} t^{1/2}\right]+\frac 1 2 \right)
\eeq
\beq N_2(t)=\frac 3 4 N_T(t)-N_{MM}(t)-\left(\left[\frac {t^{1/2}} {2\pi} \right]+\frac 1 2 \right)-\frac 1 4
\eeq
where $N_{MM}$ is given by (3.10) with $a=\frac 1 2$.\\
(c) For the square with Dirichlet boundary conditions $N_{+-}(t)$ is given by (7.4), $N_{--}(t)$ is given by (7.6),
\beq N_{+-}(t)=\frac 1 2 N_{MM}(t)+\frac 1 2 \left(\left[\frac {\sqrt 2 t^{1/2}} {4\pi} \right]+\frac 1 2 \right)
\eeq
\beq N_{-+}(t)=\frac 1 2 N_{MM}(t)-\frac 1 2 \left(\left[\frac {\sqrt 2 t^{1/2}} {4\pi} \right]+\frac 1 2 \right)
\eeq
\beq N_2(t)=\frac 3 4 N_T(t)-N_{MM}(t)-\left(\left[\frac {t^{1/2}} {2\pi} \right]+\frac 1 2 \right)-\frac 1 4
\eeq
\end{lemma}
\begin{proof} Direct substitution of triangle formulas.\end{proof}

We now define the refined asymptotics:\\
(a) For the square torus
\beq \widetilde N_{++}(t)=\frac 1 {32\pi} t + \left(\frac 1 {4\pi}+\frac {\sqrt 2} {8\pi}\right)t^{1/2}+\frac 3 8
\eeq
\beq \widetilde N_{+-}(t)=\frac 1 {32\pi} t + \left(-\frac 1 {4\pi}+\frac {\sqrt 2} {8\pi}\right)t^{1/2}-\frac 1 8
\eeq
\beq \widetilde N_{-+}(t)=\frac 1 {32\pi} t + \left(\frac 1 {4\pi}-\frac {\sqrt 2} {8\pi}\right)t^{1/2}-\frac 1 8
\eeq
\beq \widetilde N_{--}(t)=\frac 1 {32\pi} t + \left(-\frac 1 {4\pi}-\frac {\sqrt 2} {8\pi}\right)t^{1/2}+\frac 3 8
\eeq
\beq \widetilde N_2(t)=\frac  1 {8\pi}t-\frac 1 2;
\eeq
(b) For the square with Neumann boundary conditions $\widetilde N_{++}$ is given by (7.14), $\widetilde N_{-+}$ is given by (7.16),
\beq \widetilde N_{+-}(t)=\frac 1 {32\pi} t + \frac {\sqrt 2} {8\pi} t^{1/2}
\eeq
\beq \widetilde N_{--}(t)=\frac 1 {32\pi} t - \frac {\sqrt 2} {8\pi} t^{1/2}
\eeq
\beq \widetilde N_2(t)=\frac  1 {8\pi}t-\frac 1 {2\pi} t^{1/2}-\frac 1 4;
\eeq
(c) For the square with Dirichlet boundary conditions $\widetilde N_{+-}(t)$ is given by the right side of (7.15), $\widetilde N_{--}(t)$ is given by the right side of (7.17), $\widetilde N_{++}(t)$ is given by (7.19), $\widetilde N_{-+}(t)$ is given by (7.20),
\beq \widetilde N_2(t)=\frac  1 {8\pi}t+\frac 1 {2\pi} t^{1/2}-\frac 1 4.
\eeq
\indent Then the analog of Theorem 2.2 holds.\\

Next we consider the case of the hexagonal torus associated to the hexagon in Figure 2.1.  The group $D_3$ of order six is generated by the three reflections in the three bisectors.  It has two 1-dimensional representations denoted $1\pm$ according to the symmetry (7.1) with respect to all three reflections, and a 2-dimensional representation denoted 2.  The analog of (7.2) is 
\beq N_+(t)+N_-(t)+N_2(t)=N_T(t),
\eeq
and we use this to solve for $N_2(t)$.  Since $1+$ eigenfunctions may be identified with Neumann eigenfunctions on the equilateral triangle fundamental domain (shaded in Figure 2.1), and $1-$ eigenfunctions with Dirichlet eigenfunctions, we have $N_+(t)$ given by (2.22) and $N_-(t)$ by (2.23).  This leads to
\beq \widetilde N_+(t)=\frac 1 {4\pi}\frac {\sqrt 3} 4 t + \frac 3 {4\pi} t^{1/2}+\frac 1 3
\eeq
\beq \widetilde N_-(t)=\frac 1 {4\pi}\frac {\sqrt 3} 4 t - \frac 3 {4\pi} t^{1/2}+\frac 1 3
\eeq
\beq \widetilde N_2(t)=\frac {\sqrt 3} {4\pi} t - \frac 2 3
\eeq

Finally we consider the equilateral triangle with either Neumann or Dirichlet boundary conditions.  A fundamental domain for the action of $D_3$ is a $30^\circ-60^\circ-90^\circ$ triangle equal to a sixth of the equilateral triangle.  Again, the eigenfunctions on the equilateral triangle with $1\pm$ symmetry correspond to eigenfunctions on the fundamental domain with boundary conditions shown in Figure 7.2.\\

\begin{figure}[h!]
\centering
\includegraphics[scale=.6]{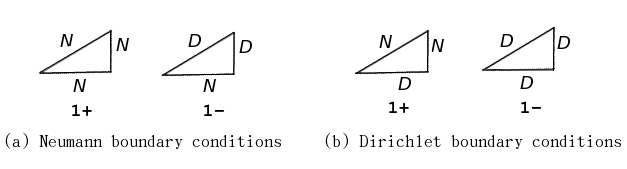}\caption{}
\end{figure}
\noindent Thus for Neumann boundary conditions, $N_T(t)$ is given by (2.26) and $N_-(t)$ is given by (3.32), while for Dirichlet boundary conditions $N_+(t)$ is given by (3.31) and $N_-(t)$ is given by (2.27).  (Here we have to rescale because the triangles are smaller by a factor of $\frac 1 {\sqrt 3}$.) This leads to the following choices:\\
\indent(a) For Neumann boundary conditions
\beq \widetilde N_+(t)=\frac 1 {4\pi}\frac {\sqrt 3}{24} t + \frac {\sqrt 3 + 1}{8\pi}t^{1/2}+\frac 5 {12}
\eeq
\beq \widetilde N_-(t)=\frac 1 {4\pi}\frac {\sqrt 3}{24} t + \frac {-\sqrt 3 + 1}{8\pi}t^{1/2}-\frac 1 {12}
\eeq
\beq \widetilde N_2(t)=\frac 1 {4\pi}\frac {\sqrt 3}{6} t -\frac 1 {4\pi} t^{1/2}-\frac 1 3;
\eeq
\indent(b) For Dirichlet boundary conditions
\beq \widetilde N_+(t)=\frac 1 {4\pi}\frac {\sqrt 3}{24} t + \frac {\sqrt 3 - 1}{8\pi}t^{1/2}-\frac 1 {12}
\eeq
\beq \widetilde N_-(t)=\frac 1 {4\pi}\frac {\sqrt 3}{24} t - \frac {\sqrt 3 + 1}{8\pi}t^{1/2}+\frac 5 {12}
\eeq
\beq \widetilde N_2(t)=\frac 1 {4\pi}\frac {\sqrt 3}{6} t +\frac 1 {4\pi} t^{1/2}-\frac 1 3.
\eeq

Again the analog of Theorem 2.2 holds.

\newpage
\section*{References}

\noindent[BS] M. van den Berg and S. Srisatkunarajah, {\em Heat flow and Brownian motion for a region in $\mathbb R^2$ with a polygonal boundary}, Probab. Theory Related Fields {\bf 86} (1990), 41-52.\\

\noindent[Bu] P. Buser, {\em Geometry and spectra od compact Riemann surfaces}, Birkhauser Boston, 1992\\

\noindent[G] P.B. Gilkey, {\em Asymptotic formulae in spectral geometry}, Chapman \& Hall /CRC, Boca Raton 2004\\

\noindent[GKS] E. Greif, D. Kaplan and R. Strichartz, {\em Spectrum of the Laplacian on regular polyhedral surfaces}, in preparation.\\

\noindent[JS] S. Jayakar and R. Strichartz, {\em Average number of lattice points in a disk}, preprint.\\

\noindent[K]  M. Kac, {\em Can one hear the shape of a drum} Amer. Math. Monthly, {\bf 783} (1966), 1-23.\\

\noindent[Sa] P. Sarnak, {\em Spectra of hyperbolic surfaces,} Bull. Amer. Math. Soc. {\bf 40} (2003), 441-478.\\

\noindent[S] R. Strichartz, {\em Spectral asymptotics revisited}, J. Fourier Anal. Appl. {\bf 18} (2012), 626-659.\\

\noindent[Ta] R. Takahashi, {\em Sur les repr\'esentations unitaires des groupes de Lorentz g\'en\'eralis\'es}, Bull. Math. Soc. Fr. {\bf 91} (1963), 289-433.\\
\end{document}